\documentclass[11pt]{article}

\usepackage{sectsty}
\usepackage{graphicx}
\usepackage{amsmath,amssymb}
\usepackage{amsthm}

\newcommand{\Exp}[1]{{\mathbb{E}}\left[#1\right] }    

\usepackage{caption}
\usepackage{subcaption}

\newcommand{\R}{\mathbb{R}} 

\usepackage{mdframed} 
\usepackage{thmtools}

\usepackage{mdframed} 
\newcommand{\myNum}[1]{(\emph{#1})}

\usepackage{soul}

\usepackage{etoolbox}
\usepackage{pgfplotstable}
\usepackage{pgfplots}
\usepackage{booktabs}

\usepackage{enumitem}

\usepackage{color,colortbl}
\definecolor{Gray}{gray}{0.2}
\usepackage{graphicx}
\usepackage{hyperref}

\usepackage{nccmath}

\usepackage{xcolor} 
\usepackage{tikz}
\usetikzlibrary{fit,calc}

\colorlet{mypink}{red!40}
\colorlet{myblue}{cyan!40}

\usepackage{xcolor}
\usepackage{color,colortbl}
\definecolor{Gray}{gray}{0.75}
\usepackage{graphicx}
\usepackage{url}

\newtheorem{theorem}{Theorem}  
\newtheorem{corollary}{Corollary} 
\newtheorem{remark}{Remark} 

\usepackage{amsthm}

\usepackage{mdframed} 
\newcommand{\smartparagraph}[1]{\vspace{2pt} \noindent {\bf #1}}
\date{}

\usepackage{caption}

\usepackage{booktabs} 

\usepackage[colorinlistoftodos,bordercolor=orange,backgroundcolor=orange!20,linecolor=orange,textsize=scriptsize]{todonotes}

\title{Where Have All the Kaczmarz Iterates Gone?}
\author{El Houcine Bergou~\thanks{College of Computing, Mohammed VI Polytechnic University, Morocco, elhoucine.bergou@um6p.ma.}
\and
Soumia Boucherouite~\thanks{College of Computing, Mohammed VI Polytechnic University, Morocco, soumia.boucherouite@um6p.ma.}
\and
Aritra Dutta \thanks{Department of Mathematics, University of Central Florida, USA, {aritra.dutta@ucf.edu}.}
\and
Xin Li\thanks{Department of Mathematics, University of Central Florida, USA, {xin.li@ucf.edu}.} 
\and
Anna Ma \thanks{Department of Mathematics, University of California, Irvine, USA, {anna.ma@uci.edu}.}
}

\begin{document}
\maketitle  

\begin{abstract}
The randomized Kaczmarz (RK) algorithm is one of the most computationally and memory-efficient iterative algorithms for solving large-scale linear systems. However, practical applications often involve noisy and potentially inconsistent systems. While the convergence of RK is well understood for consistent systems, the study of RK on noisy, inconsistent linear systems is limited. This paper investigates the asymptotic behavior of RK iterates in expectation when solving noisy and inconsistent systems, addressing the locations of their limit points. We explore the roles of singular vectors of the (noisy) coefficient matrix and derive bounds on the convergence horizon, which depend on the noise levels and system characteristics. Finally, we provide extensive numerical experiments that validate our theoretical findings, offering practical insights into the algorithm's performance under realistic conditions. These results establish a deeper understanding of the RK algorithm's limitations and robustness in noisy environments, paving the way for optimized applications in real-world scientific and engineering problems.
\end{abstract}


\textbf{AMS subject classes.}
15A06, 15A09, 15A10, 15A18, 65F10, 65Y20, 68Q25, 68W20, 68W40

\section{Introduction}

Solving systems of linear equations 
is one of the most common and fundamental tasks in many areas of science and engineering \cite{Bjrock1996, dutta2019online, chiu2014efficient, FRANK2016, WANG2019}. As problem sizes grow, computational and storage requirements increase significantly. Consequently, direct solvers such as LU decomposition, Cholesky, or QR factorization, which operate on the entire data matrix $A$, become too expensive and impractical due to their expensive memory usage and excessive floating-point operations \cite{Bjrock2015, saad2000, Bjrock1996, saad2003, Greenbaum1997}. In contrast to direct solvers, iterative methods \cite{needell2010randomized, ma2015convergence, sahu2020convergence} do not always require storing the entire matrix $A$ in memory. They start with an initial approximation, progressively refine it in successive iterations, and continue until an approximate solution of the desired accuracy is achieved. One such iterative method for solving large-scale consistent linear systems is the Kaczmarz algorithm \cite{Kaczmarz1937}. The Kaczmarz algorithm is a {\em row-action method} as it operates with only one row of $A$ at each iteration, making it compute- and memory-efficient. 

Consider a system of linear equations:
\begin{equation}\label{prb:consistent}
    Ax = b,
\end{equation}
where $A \in \R^{m\times n}$ and $b \in \R^m$ are given, and $x \in \R^n$ is an unknown vector. Let $a_i^\top$ denote the $i$-th row of $A$. The Kaczmarz algorithm starts from an initial point, $x_0$, and iterates by projecting the current estimate $x_k$ onto the solution space of a single equation, $a_{i(k)}^\top x=b_{i(k)}$, where 
$i(k)$ is selected cyclically from $\{1,2,...,m\}$. {More precisely, given an initialization $x_0$, the iterates are defined as follows}:

\begin{align}\label{eq:ka}
x_{k+1}=x_k-\frac{a_{i(k)}^\top x_k-b_{i(k)}}{\| a_{i(k)}\|^2} a_{i(k)},~k=0,1,2,\cdots,
\end{align}
where $i(k)=(k$ mod $m) + 1$. 

The Kaczmarz algorithm is often referred to as {\em cyclic Kaczmarz}. Kaczmarz~\cite{Kaczmarz1937} originally considered systems with square matrices and proved that when $A$ is nonsingular, the sequence of Kaczmarz iterates $\{x_k\}$ converges to the unique solution. However, \cite{Kaczmarz1937} did not provide a convergence rate. 

While it was observed that the convergence rate of {cyclic Kaczmarz} could suffer when consecutive rows of $A$ were close to identical, empirical observations suggested that random row selection accelerated the convergence by avoiding such worst-case orderings \cite{Feichtinger1992, Herman1993}. 
Strohmer and Vershynin \cite{Strohmer2013} validated these empirical findings by introducing the randomized Kaczmarz algorithm (RK), where at the $k$-th iteration, $i(k)$ is randomly selected from $\{1,2,...,m\}$ with probability proportional to the row norm of $a_{i(k)}$. They demonstrated that RK converges linearly in expectation to the unique solution of consistent, overdetermined linear systems, as follows.

\begin{theorem}[\cite{Strohmer2013}, Theorem 2]
\label{thm:SV}
Assume that (\ref{prb:consistent}) is consistent and $A$ is of full column rank. Let $x\in \R^n$ be the unique solution, and the initial point $x_0\in \R^n$ be arbitrary.

Let $i(k)$ be chosen from $\{1,2,...,m\}$ at random, with 
${\rm Prob}(i(k)=i)=\frac{\|a_{i}\|^2}{\|A\|_F^2}$ where $\|A\|_F$ denotes the Frobenius norm of matrix $A$. Then the sequence of the iterates $\{x_k\}$ \eqref{eq:ka} converges to the unique solution $x$ in expectation, and the expected error satisfies
\begin{equation*}
    \mathbb{E} \|x_k-x\|^2\leq \left(1-\frac{\sigma_{\rm min}^2}{\|A\|^2_F}\right)^k\|x_0-x\|^2,
\end{equation*}
where $\sigma_{\rm min}$ is the smallest singular value of $A$.
\end{theorem}  

Because noise is inevitable in practice due to the data-curation process and many other artifacts~\cite{machiels1998numerical,keller1964stochastic, du2002numerical, gajek2021errors,chiu2014efficient}, while solving for (\ref{prb:consistent}), we may only be given noisy versions of $A$ and $b$: $\widetilde{A}=A+E$ and $\widetilde{b}=b+\epsilon$, where $E \in \R^{m\times n}$ and $\epsilon \in \R^m$. Thus, it is essential to understand Kaczmarz iterates of the form
\begin{equation}
    \label{eq:rknoisy}
    x_{k+1}=x_k-\frac{\tilde{a}_{i(k)}^\top x_k-\tilde{b}_{i(k)}}{\| \tilde{a}_{i(k)}\|^2} \tilde{a}_{i(k)},~k=0,1,2,\cdots,
\end{equation}
when applied to
\begin{equation} \label{prb:inconsistent}
    \widetilde{A}x \approx \widetilde{b}.
\end{equation}

Generally, noisy linear systems \eqref{prb:inconsistent} are not guaranteed to be consistent. Without the consistency, we cannot apply the convergence results of Kaczmarz or Strohmer and Vershynin \cite{Strohmer2013}. One must wonder if the Kaczmarz algorithm is applied to an inconsistent system, where the Kaczmarz iterates would go.

In contrast to changing the algorithm for inconsistent systems, such as the randomized extended Kaczmarz algorithm~\cite{zouzias2013randomized}, this work focuses on understanding the behavior of the classical randomized Kaczmarz algorithm applied to noisy systems \eqref{prb:inconsistent}. When $E=0$ and $A$ is of full column rank, Needell \cite{needell2010randomized} proved that the iterates of RK applied to \eqref{prb:inconsistent} approach a ball centering at the least squares solution $x_{\rm LS}$ in expectation, where the radius of the ball is given by using the bound on the noise, $\|\epsilon\|_{\infty}$.~Zouzias and Freris \cite{zouzias2013randomized} dropped the full rank assumption and further {generalized} the results of \cite{Strohmer2013} and \cite{needell2010randomized}. For convenience, we restate the result of Zouzias and Freris below. 
 
\begin{theorem}[\cite{zouzias2013randomized}, Theorem 2.1]\label{thm:zouzias} Let $x_k$ be the RK iterates \eqref{eq:rknoisy} when RK algorithm is
applied the linear system~\eqref{prb:inconsistent} with $E=0$. Let $x_0\in {\rm range}(A^\top)$ and $x_{\rm LS}=A^{\dagger}b$. Then, the sequence $\{x_k\}$ satisfies
\begin{equation*}
   \mathbb{E}\|x_k-x_{\rm LS}\|^2\leq \left(1-\frac{\sigma_{\rm min}^2}{\|A\|^2_F}\right)^k\|x_0-x_{\rm LS}\|^2+\frac{\|\epsilon\|^2}{\sigma_{\rm min}^2}. 
\end{equation*}
where $\sigma_{\rm min}$ is the smallest non-zero singular value of $A$.
\end{theorem}

We remark that the guarantees shown in Theorem~\ref{thm:SV} can be seen as a special case of Theorem~\ref{thm:zouzias} by taking $\epsilon=0$. While previous works only consider noise on the right-hand side (i.e., with $E = 0$),
Bergou et al.~\cite{doubly_noisy_RK} provide convergence analysis of RK in general noisy cases, including $E\neq 0$ and $\epsilon\neq 0$. Their result can be stated as follows.
\begin{theorem} [\cite{doubly_noisy_RK}, Theorem 3.1] \label{thm:bergou} {Let $A x = b$ be a fixed consistent system with solution } $x_{\rm LS}= A^{\dagger}b$ and let $x_0-x_{\rm LS}\in {\rm range}(\widetilde A^\top)$. Then, the sequence $\{x_k\}$ obtained in (\ref{eq:rknoisy}) by applying the RK algorithm to \eqref{prb:inconsistent} {where $E = \tilde{A} - A$ and $\epsilon = \tilde{b} - b$ satisfies}
\begin{equation*}
  \mathbb{E}\|x_k-x_{\rm LS}\|^2 \leq \left(1-\frac{\widetilde{\sigma}_{\rm min}^2}{\|\widetilde{A}\|^2_F}\right)^k \|x_0-x_{\rm LS}\|^2+\frac{\|Ex_{\rm LS}-\epsilon\|^2}{\widetilde{\sigma}_{\rm min}^2}.   
\end{equation*}
where $\widetilde{\sigma}_{\rm min}$ is the smallest non-zero singular value of $\widetilde A$.
\end{theorem}
Since the convergence of $\{x_k\}$ implies the limit must be a solution to the system, it follows that the Kaczmarz iterates $\{x_k\}$ fail to converge if the underlying system is inconsistent.
Following Needell~\cite{needell2010randomized}, all works on noisy systems have used the LS solution $x_{\rm LS}=A^\dagger b$ of the underlying consistent noiseless system as a reference point for the limiting behavior of the Kaczmarz iterates.
This serves the purpose of measuring how far the Kaczmarz iterates of the noisy system are from the limit of the corresponding consistent system: they are eventually approaching a ball centered at $x_{\rm LS}$ and of radius given by the level of noise terms. 
We will refer to this radius as the {\it convergence horizon}, or simply {\it horizon}, following \cite{needell2010randomized} and \cite{ma2015convergence}. But, in terms of locating all the limiting points of the Kaczmarz iterates, $x_{\rm LS}$ is not the best point to use as the center of a ball that contains all the limit points. Indeed, we will show that $\widetilde{x}_{\rm LS}=\tilde{A}^\dagger \tilde{b}$, the LS solution of the noisy system, is the center that yields the smallest horizon. Figure~\ref{fig:circles1} shows one such scenario---the plot on the left shows three possible locations of $x_{\rm LS}$ for three possible consistent systems, but their radius are all larger than the one corresponding to $\widetilde{x}_{\rm LS}$.

{Another issue with Theorems 1.2 and 1.3 is their requirements on the initial points.} The above results require $x_0\in {\rm range}(A^\top)${ when $E=0$. This can be easily satisfied by taking $x_0 = 0$ or $x_0 = a_i$ for some row $i$. Unfortunately, this is not true for the case when $E \neq 0$, which requires} $x_0-x_{\rm LS}\in {\rm range}({\widetilde{A}}^\top)$, {as $x_{\rm LS}$ is unknown apriori}.

\begin{figure}[t]
    \centering
    \includegraphics[width=0.45\textwidth]{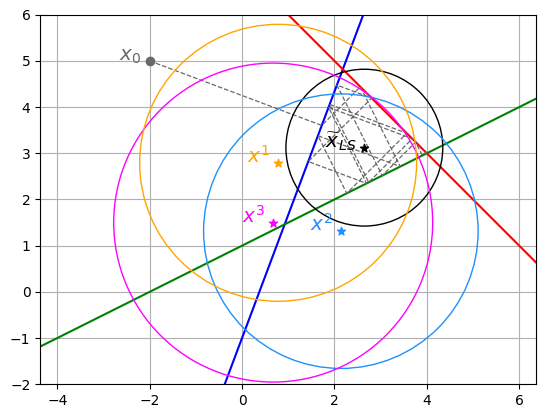}
    \includegraphics[width=0.45\textwidth]{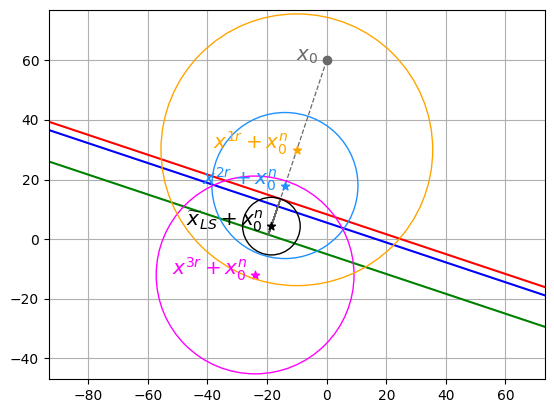}
    \caption{\textbf{Path of the RK iterates.} RK is applied to $\widetilde Ax \approx \widetilde b$ for $\widetilde A \in \R^{3\times 2}$. {We consider the cases in which (left) ${\rm rank}(\widetilde A)=2$ and $x_0,x^1,x^2,x^3 \in {\rm range}(\widetilde A ^\top)$ and (right) ${\rm rank}(\widetilde A)=1$ and $x_0,x^1,x^2,x^3 \notin {\rm range}(\widetilde A ^\top)$.} {Each line represents a solution space determined by a row of the linear equalities }$\widetilde Ax \approx \widetilde b$. {The point} $\widetilde x_{\rm LS}=\widetilde A^\dagger \widetilde b$ and $x^1,x^2,x^3$ are arbitrary. The circles are of centers $x_0^n+x_*^r$ and radius $\|\widetilde A \widetilde x_* - \widetilde b \|/\widetilde{\sigma}_{min}$ for $x_*=\widetilde x_{\rm LS},x^1,x^2,x^3$. The intersection of all the circles is going to be where all the cluster points are.} 
  \label{fig:circles1}
\end{figure}

We will analyze the limit points of RK iterates for an arbitrary initialization when applied to the noisy system \eqref{prb:inconsistent}. In doing so, we accomplish two feats. First, we remove the assumption that $x_0-x_{\rm LS}\in {\rm range}({\widetilde{A}}^\top)$ by analyzing the limiting behavior to an arbitrary reference point. Second, we characterize the ball of smallest horizon to demonstrate that the optimal horizon is achieved by setting the reference point to be the least squares solution of the noisy system. This implies that although there may be an underlying consistent system $Ax = b$ which has been perturbed to $\widetilde{A} x = \widetilde{b}$, without additional adaptations of the algorithm, in general, the best one can hope for is the least squares solution of the noisy system. However, this also shows that there are special cases in which one can still attain the solution to the original consistent system $Ax = b$ even if RK is applied to \eqref{prb:inconsistent}. In particular, for specific choices of noise $E$ and $\epsilon$, the convergence horizon can be avoided altogether.

Taken together, our contributions are summarized as follows:

\myNum{i} \smartparagraph{Convergence analysis for RK with respect to arbitrary reference points and arbitrary noise in data.} 
In {Section~\ref{sec:refpoints}}, we {show a generalization of} Theorems~\ref{thm:zouzias} and \ref{thm:bergou} that will de-emphasize the role of any particular underlying consistent system 
and accommodate an arbitrary initial point $x_0$; see a summary in Table \ref{table1}.~We also present a refined convergence along singular vectors analysis using the Steinerberger linearized method~\cite{steinerberger2021randomized} in Section~\ref{sec:singular_vectors}.

\myNum{ii}\smartparagraph{Characterization of the limiting points of the iterates of RK.} We leverage the arbitrary reference point analysis to show that, for a fixed perturbed system, the ball of the smallest radius is attained via the least squares solution to the noisy system. These results are presented in Section~\ref{sec:bounding_balls}. {Lastly, in {Section~\ref{sec:numerics}}, we present numerical experiments that validate our theoretical results and demonstrate improved convergence bounds in special cases.}

\subsection{Notations}\label{notations}
{A vector} $x$ is always referring to a column vector with $x^\top$ as its transpose, and $\|x\|$ denotes its $\ell_2$-norm. 
For a matrix, $M \in \R^{m\times n}$, $M^\top$ and $M^\dagger$ are the transpose and the pseudo{-}inverse of $M$, respectively. 
We denote the spectral and Frobenius norms of $M$ by $\|M\|$ and $\|M\|_F$, respectively. 
The condition number of $M$ is denoted as $\kappa(M)$. 

We denote $\sigma_{\rm min}=\sigma_r$ as the smallest {nonzero} singular value and $\sigma_1$ as the largest singular value of {$A$}, where $r$ is the rank of {$A$}. {To distinguish between the singular values of the noiseless matrix $A$ and noisy matrix $\widetilde{A}$, we use} $\widetilde{\sigma}$, with appropriate subscript.
For any vector $x\in \R^n$, we denote by $v^r$ and $v^n$ the the orthogonal projections of vector $v$ into ${\rm range}(\widetilde A^\top)$ and its orthogonal complement ${\rm null}(\widetilde A)$, respectively.
For a linear system, $Ax=b$, $x_{\rm LS}$ denotes the least squares (LS) solution with the minimum norm, and we have $x_{\rm LS} = A^\dagger b$. {Lastly, we simplify the notation for the sequence of RK iterates $\{x_k\}_{k=0}^\infty$ to $\{x_k\}$ throughout.}

\def\buildTable#1{%
    \pgfplotstableread[col sep = comma]{#1}\rawdata%
    \pgfplotstabletypeset[before row= \midrule]\rawdata
}%

\begin{table}
\footnotesize
\centering
\setlength{\tabcolsep}{1.72pt}
\begin{tabular}{c c c c c c}  

\multicolumn{6}{c}{} \\
 \toprule
         &                 &  &{Convergence} & {Convergence}  & \\
Quantity & Linear system   & Assumptions & rate          & horizon        & Reference\\
 \toprule
 
$\mathbb{E}[\|x_k-x_{\rm LS}\|^2]$ & $Ax\approx \tilde{b}$ & $A$ is full-rank & $(1-\frac{1}{R})$ & $R (\max_i \frac{\epsilon_i}{\|A_i\|})^2 $ & \cite{needell2010randomized},\\
& & & & & \scriptsize{Theorem 2.1}\\
\midrule

$\mathbb{E}[\|x_k-x_{\rm LS}\|]$ & $Ax\approx \tilde{b}$ & $x_0 \in {\rm range} (A^\top)$ & $(1-\frac{1}{R})$ & $\frac{\|\epsilon \|^2}{\sigma_{\rm min}^2}$ & \cite{zouzias2013randomized}, \\
& & & & & \scriptsize{Theorem 2.1}\\
\midrule

$\mathbb{E}[\|x_k-x_{\rm LS}\|^2]$ & $\tilde{A} x \approx \tilde{b}$ &  $x_0\!- \!x_{\rm LS} \!\in \!{\rm range} ( \tilde{A}^\top \!)$ & $(1-\frac{1}{\tilde{R}})$ & $\frac{\|Ex_{LS} - \epsilon \|^2}{\tilde{\sigma}_{\rm min}^2}$ & \cite{doubly_noisy_RK},\\
& & & & & \scriptsize{Theorem 3.1}\\
\midrule
\rowcolor{Gray}
$\mathbb{E}[\|x_k-x_0^n-x_*^r\|^2]$ &  $ \tilde{A} x \approx \tilde{b}$ & No assumptions & $(1-\frac{1}{\tilde{R}})$ & $\frac{\| \tilde{A} x_{*} - \tilde{b} \|^2}{\tilde{\sigma}_{\rm min}^2}$ &  \scriptsize{This work,}\\ 
\rowcolor{Gray}& & & & & \scriptsize{Theorem \ref{thm:DoublyNoisy_x}}\\
\bottomrule


\end{tabular}
\caption{\textbf{Summary of RK applied to noisy linear systems.} Here $A$ and $b$ are the true matrix and vector, respectively, and $Ax = b$ is consistent, $ \tilde{A}=A+E$ and $ \tilde{b}=b+\epsilon$ are the noisy data, $R= \|A \|^2_F/\sigma_{\min}^2$, and
$\tilde{R}= \|\tilde{A} \|^2_F/\tilde{\sigma}_{\min}^2$ with  $\sigma_{\rm min}$ and $\tilde{\sigma}_{\rm min}$ being the smallest non-zero singular values of $A$ and $\tilde{A}$ respectively. The vectors, $x_0, x_* \in \R^n$, are arbitrary in the last row.}
\label{table1}
\end{table}


\section{Using arbitrary reference points}
\label{sec:refpoints}

This work takes two novel perspectives. In particular, previous works analyzing doubly noisy systems assume that there is an underlying consistent system that is perturbed by noise. In this work, we adopt a different perspective: we start with a possibly inconsistent system \eqref{prb:inconsistent}, then choose a consistent system \eqref{prb:consistent}, and we define $E$ and $\epsilon$ by 
\begin{equation*}
  E=\widetilde{A}-A~~{\rm and}~~\epsilon=\widetilde{b}-b.  
\end{equation*}
From this perspective, we can focus on possible limit points (see Theorem~\ref{thm:DoublyNoisy_x} and \ref{thm:singular value additive noisy case}) and ask for the best choices of $A$ and $b$ (see Theorem~\ref{thm:smallest-ball} and Corollary~\ref{thm:smallest-ball2}).

\begin{figure}
    \centering
    \begin{subfigure}[t] 
    {0.5\textwidth}
        \centering
    \includegraphics[width=\textwidth]{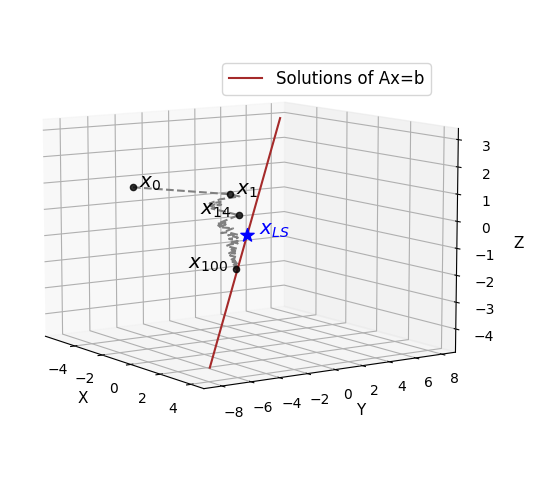}
    \caption{}
    \label{subfig_a}
    \end{subfigure}%
    \begin{subfigure}[t]{0.45\textwidth}
        \centering
    \includegraphics[width=\textwidth]{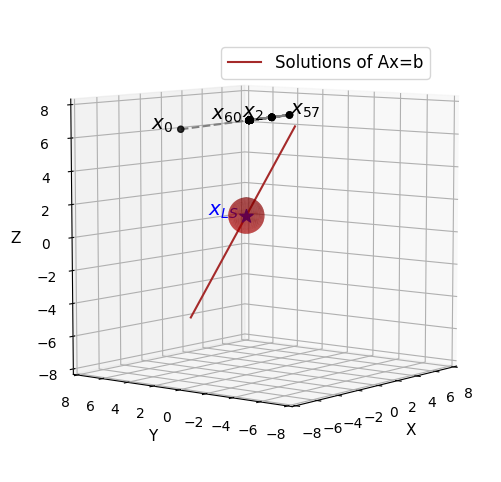}
    \caption{}
    \label{subfig_b}
    \end{subfigure}%
    \caption{\textbf{Path of RK iterates.} The system, $Ax=b$ is consistent with $A \in \R^{6 \times 3}$ of {rank $2$} and $x_{\rm LS}=A^\dagger b$. 
    (a) RK applied to $Ax = b$ with $x_0 \notin {\rm Range} ({A}^\top)$. 
    (b) RK applied to $\tilde{A}x\approx \tilde{b}$ with $\tilde{A} = A+E$, $\tilde{b} = b+\epsilon$, and $x_0 - x_{\rm LS} \notin {\rm Range} ( \tilde{A}^\top)$. The radius of the ball centered around $x_{\rm LS}$ is $\|Ex_{\rm LS}-\epsilon\|/{\tilde{\sigma}_{\rm min}}$.\vspace{-1.2em}}
    \label{fig:path_noiseFree_notinRange}
\end{figure}

While the literature focuses on considering the least squares solution $x_{\rm LS}$ as the reference point for the RK iterates, RK does not always converge to this point, particularly for noisy or inconsistent systems. Additionally, there are assumptions on the initial points. For example, we note that Theorems~\ref{thm:zouzias} and \ref{thm:bergou} require either $x_0\in {\rm range} (A^\top)$ {when $E = 0$}  or $ x_0-x_{\rm LS} \in {\rm range} (\widetilde A^\top)$, where $x_{\rm LS}$ is the least squares solution of the underlying consistent system, {when $E \neq 0$}. These requirements are necessary for estimates as demonstrated by Figures~\ref{subfig_a} and \ref{subfig_b}. They show the paths of the iterates of RK applied to simple examples. In the consistent case ($E=0$ and $\epsilon = 0$), the iterates converge to a limit different from $x_{\rm LS}$; in the inconsistent case, the convergence horizon given by Theorem~\ref{thm:bergou} is represented by a ball centered around $x_{\rm LS}$, but still the iterates $\{x_k\}$ stay strictly outside the ball.

We need to remove the requirements on the initialization point $\{x_0\}$ so that we can
introduce an arbitrary reference point for RK iterates $\{x_k\}$.
Theorem~\ref{thm:DoublyNoisy_x0} below presents a simple formulation incorporating the initial point into the reference point for the RK iterates, which describes the clustering of the RK iterates to the reference point $x_0^n + x_{\rm LS}^r$, where $x_0^n$ denotes the projection onto the null space of $x_0$ and $x_{\rm LS}^r$ the projection of {$x_{\rm LS}$} onto the row space of $\widetilde{A}$,  
up to a horizon depending on the noise in $A$ and $b$. 

\begin{theorem}\label{thm:DoublyNoisy_x0}
Let $\{x_k\}$ be the iterates obtained  in (\ref{eq:rknoisy}) when the RK algorithm is applied to~\eqref{prb:inconsistent} where $\widetilde{A}$ and $\widetilde{b}$ are fixed. Let $x_0$ be arbitrary. We have, {for any $A \in \mathbb{R}^{m \times n}$ and any $b \in \mathbb{R}^{m}$,}
\begin{equation}\label{rate1_x0}
 \mathbb{E}\| x_{k} -x_0^n- x_{\rm LS}^r\|^2 \leq \left(1-\frac{1}{\widetilde{R}}\right)^{k}  \| x_0^r-x_{\rm LS}^r \|^2+ \frac{\| E x_{\rm LS} - \epsilon \|^2}{\widetilde{\sigma}_{\rm min}^2}, 
\end{equation}    
where $\widetilde{R} = \|\widetilde A\|^2_F/\widetilde{\sigma}_{\rm min}^2$, $x_{\rm LS}=A^\dagger b$, {$E = \widetilde{A} - A$ and $\epsilon = \widetilde{b} - b$}.
\end{theorem}

As it turns out, this follows from a more general result given in the following theorem,
which considers clustering to the projection of an arbitrary reference point $x_*$.
\begin{theorem}\label{thm:DoublyNoisy_x}
Let $\{x_k\}$ be the iterates obtained  in (\ref{eq:rknoisy}) when the RK algorithm is applied to~\eqref{prb:inconsistent} where $\widetilde{A}$ and $\widetilde{b}$ are fixed. 
Let $x_0$ be arbitrary. We have, for any $x_* \in \R^n$,
\begin{equation}\label{rate1}
 \mathbb{E}\| x_{k} -x_0^n- x_*^r\|^2 \leq \left(1-\frac{1}{\widetilde{R}}\right)^{k} \| x_0^r-x_*^r \|^2+ \frac{\| \widetilde{A} x_* - \widetilde{b} \|^2}{\widetilde{\sigma}_{\rm min}^2}, 
\end{equation}    
where $\widetilde{R} = \|\widetilde A\|^2_F/\widetilde{\sigma}_{\rm min}^2$.
\end{theorem}

\begin{proof}
The proof 
    can be obtained by \myNum{i} observing that the same argument as given in the proof of \cite[Thm. 3.1]{doubly_noisy_RK} can be used to show that $x_{\rm LS}$ there can be replaced by an arbitrary reference point $x_*$, and \myNum{ii} taking $A:=\widetilde{A}$ and $b:= \widetilde{A}x_* = Ax_*$ in applying
    \cite[Thm. 3.1]{doubly_noisy_RK} as we modified it in \myNum{i}. We leave the details to the reader.
\end{proof}

\section{Analysis of the singular vectors effect} \label{sec:singular_vectors}

{It was recently shown that when} RK {is} applied to consistent linear systems with full-rank matrices, the sequence $\{x_k\}$ does not converge to $x_{\rm LS}$ randomly ``from all possible directions"  but rather, it follows a specific pattern~\cite{steinerberger2021randomized}. The iterates of RK start by approximating $x_{\rm LS}$ from directions described by the right singular vectors corresponding to the largest singular values. For $k$ large, $x_k-x_{\rm LS}$ will mainly become a combination of right singular vectors corresponding to small singular values, and convergence to $x_{\rm LS}$ becomes slower. 
Particularly, the exponential decrease of the approximation error $\|x_k-x_{\rm LS}\|$ happens at different rates in different subspaces.  When the iterate $x_k$ is not mainly the linear span of right singular vectors corresponding to small singular values, RK enjoys a faster convergence rate. More precisely, Steinerberger derived the following result.

\begin{theorem} [\cite{steinerberger2021randomized}, Theorem 1] \label{Steinerberger}
Let $\{x_k\}$ be the iterates of RK applied to a consistent system as in \eqref{prb:consistent}. Assume that $A$ is of full column rank. Let $v_j$ be the $j$th right singular vector of $A$ associated with the singular value $\sigma_j$. Then {for any initialization $x_0 \in \mathbb{R}^n$}:
\begin{equation}
     \mathbb{E} \langle x_{k} - x_{\rm LS},v_j \rangle = \left(1-\frac{\sigma_j^2}{\|A\|_F^2}\right)^{k} 
     \langle  x_0-x_{\rm LS},v_j \rangle.
\end{equation}
for $j=1,...,n$.
\end{theorem}

In what follows, we extend the results of Steinerberger~\cite{steinerberger2021randomized} and analyze the smallest singular vector effect when there is noise in the matrix $A$ and vector $b$. We do not impose any assumption on the initial point $x_0$ of the algorithm or on the rank of the matrix. {Our error analysis reveals a dependence on the left singular vectors of $\widetilde{A}$:}

\begin{theorem}\label{thm:singular value additive noisy case}
Let $\{x_k\}$ be the iterates obtained in (\ref{eq:rknoisy}) when the RK algorithm is applied to the doubly-noisy linear system~\eqref{prb:inconsistent}. Let $x_0\in \R^n$ be arbitrary. Let $\widetilde u_j$ and $\widetilde v_j$ be the $j$th left and right singular vectors of $\widetilde A$ associated with singular value $\widetilde{\sigma}_j$. Then, for any $x_* \in \R^n$, we have
\begin{align}   \label{eq:iterate-singular-vector}
 \mathbb{E} \langle x_{k} - x_0^n-x_*^r,\widetilde v_j \rangle = & \left(1-\frac{\widetilde{\sigma}_j^2}{\|\widetilde{A}\|_F^2}\right)^{k}  \langle  x_0^r-x_*^r,\widetilde v_j \rangle  \\
 & -   \left[1- \left(1-\frac{\widetilde{\sigma}_j^2}{\|\widetilde A\|_F^2}\right)^k\right] 
 \frac{\langle \widetilde{A} x_* - \widetilde{b}, \widetilde u_j \rangle} 
{\widetilde{\sigma}_{j}}. \nonumber
\end{align} 
\end{theorem}

\begin{proof} 
We can write 
\begin{align*}
x_{k+1}&=x_k-\frac{\widetilde{a}^\top_{i(k)}x_k-\widetilde{b}_{i(k)}}{\|\widetilde{a}_{i(k)}\|^2} \widetilde{a}_{i(k)} \\
&=x_k-\frac{\widetilde{a}^\top_{i(k)}(x_k-x_*)+\widetilde{a}^\top_{i(k)}x_*-\widetilde{b}_{i(k)}}{\|\widetilde{a}_{i(k)}\|^2} \widetilde{a}_{i(k)}.
\end{align*}
Note that $\widetilde{a}^\top_{i(k)}(x_k-x_*)=\widetilde{a}^\top_{i(k)}(x_k-x_*^r)$ and $\widetilde{a}_{i(k)}^\top x_0^n=0$, we can write
\begin{align}\label{1star}
    x_{k+1}\!-\!x_0^n\!-\!x_*^r \!=\!
    x_{k}\!-\!x_0^n\!&-\!x_*^r  
    \!-\!\frac{\displaystyle \widetilde{a}^\top_{i(k)}(x_k-x_0^n-x_*^r)}{\displaystyle \|\widetilde{a}_{i(k)}\|^2} \widetilde{a}_{i(k)}-\frac{\displaystyle \widetilde{a}_{i(k)}^\top x_*-\widetilde{b}_{i(k)}}{\displaystyle \|\widetilde{a}_{i(k)}\|^2} \widetilde{a}_{i(k)}. 
\end{align}
To simplify notation, denote {$z_{k}=x_{k}-x_0^n-x_*^r$ for all $k$.} Taking inner product with the $j$th right singular vector $\widetilde v_j$ on both sides of \eqref{1star} {we have}
\begin{align*}
      \langle z_{k+1},\widetilde v_j\rangle
    =
    \langle z_{k},\widetilde v_j\rangle
    -\frac{\widetilde{a}^\top_{i(k)}z_k}{\|\widetilde{a}_{i(k)}\|^2} \langle \widetilde{a}_{i(k)},\widetilde v_j\rangle 
 -\frac{\widetilde{a}_{i(k)}^\top x_*-\widetilde{b}_{i(k)}}{\|\widetilde{a}_{i(k)}\|^2} \langle \widetilde{a}_{i(k)},\widetilde v_j\rangle.
\end{align*}
Next, we take {the expectation, conditioned on $x_k$}, to get
\begin{align}\label{2star}
  {\mathbb E}_k(\langle z_{k+1},\widetilde v_j\rangle)
    =&
    \langle z_{k},\widetilde v_j\rangle
    -{\mathbb E}_k\left(\frac{\widetilde{a}^\top_{i(k)}z_k}{\|\widetilde{a}_{i(k)}\|^2} \langle \widetilde{a}_{i(k)},\widetilde v_j\rangle\right)\\ 
    &-{\mathbb E}_k\left(\frac{\widetilde{a}_{i(k)}^\top x_*-\widetilde{b}_{i(k)}}{\|\widetilde{a}_{i(k)}\|^2} \langle \widetilde{a}_{i(k)},\widetilde v_j\rangle \right). \nonumber
\end{align}
Note that, for calculating expectation in the last two terms, {the probability of sampling row $i$ is proportional to the row norm of the given noisy matrix $\widetilde{A}$, i.e., $p_{i(k)}=\mfrac{\|\widetilde{a}_{i(k)}\|^2}{\sum_{i(k)=1}^m\|\widetilde{a}_{i(k)}\|^2}=\mfrac{\|\widetilde{a}_{i(k)}\|^2}{\|\widetilde{A}\|^2_F}$.} 
{The second term of \eqref{2star} simplifies to:}
\begin{align*}
  {\mathbb E}_k\left(\frac{\widetilde{a}^\top_{i(k)}z_k}{\|\widetilde{a}_{i(k)}\|^2} \langle \widetilde{a}_{i(k)},\widetilde v_j\rangle\right)
&=
\sum_{i=1}^m\frac{\|\widetilde{a}_i\|^2}{\|\widetilde{A}\|_F^2} \frac{\widetilde{a}^\top_{i}z_k}{\|\widetilde{a}_{i}\|^2}  \langle \widetilde{a}_i,\widetilde v_j\rangle  
=\frac{\sum_{i=1}^m \widetilde{a}_i^\top z_k \widetilde{a}_i^\top \widetilde v_j}{\|\widetilde{A}\|_F^2} \\
&=\frac{\langle \widetilde{A} z_k, \widetilde{A} \widetilde v_j\rangle}{\|\widetilde{A}\|_F^2} 
=\frac{\widetilde \sigma_j^2 \langle z_k,\widetilde v_j\rangle}{\|\widetilde{A}\|_F^2},
\end{align*}
and {the last term of \eqref{2star} simplifies to}
\begin{align*}
   {\mathbb E}_k\!\left(\!\frac{
\widetilde{a}_{i(k)}^\top x_*\!-\!\widetilde{b}_{i(k)}}{\|\widetilde{a}_{i(k)}\|^2} \langle \widetilde{a}_{i(k)}\!,\!\widetilde v_j\rangle\!\right)\!
&=\!
\sum_{i=1}^m\!\frac{\|\widetilde{a}_i\|^2}{\|\widetilde{A}\|_F^2} \frac{\widetilde{a}_{i}^\top x_*\!-\!\widetilde{b}_{i}}{\|\widetilde{a}_{i}\|^2} \langle \widetilde{a}_{i}\!,\!\widetilde v_j\!\rangle 
   =\frac{1}{\|\widetilde{A}\|^2_F} \sum_{i=1}^m 
(\widetilde{a}_{i}^\top x_*\!-\!\widetilde{b}_{i})
\langle \widetilde{a}_{i},\widetilde v_j\rangle \\
&=  \frac{1}{\|\widetilde{A}\|^2_F} \langle \widetilde{A}x_*-\widetilde{b},\widetilde{A} \widetilde v_j\rangle  = \frac{\widetilde \sigma_j \langle \widetilde{A}x_*-\widetilde{b},\widetilde u_j\rangle}{\|\widetilde{A}\|^2_F}. 
\end{align*}
Using these in \eqref{2star}, we obtain
\begin{equation*}
      {\mathbb E}_k(\langle z_{k+1},\widetilde v_j\rangle)
    =
    \langle z_{k},\widetilde v_j\rangle
    -\frac{\widetilde \sigma_j^2 \langle z_k,\widetilde v_j\rangle}{\|\widetilde{A}\|_F^2}
    - \frac{\widetilde \sigma_j \langle \widetilde{A} x_*-\widetilde{b},\widetilde u_j\rangle}{\|\widetilde{A}\|_F^2}.
\end{equation*}
Thus,
\begin{equation*}
    {\mathbb E}(\langle z_{k+1},\widetilde v_j\rangle )
    =\left(1
    -\frac{\widetilde \sigma_j^2 }{\|\widetilde{A}\|_F^2}\right)
    {\mathbb E}(
    \langle z_{k},\widetilde v_j\rangle )
    - \frac{\widetilde \sigma_j \langle \widetilde{A} x_*-\widetilde{b},\widetilde u_j\rangle}{\|\widetilde{A}\|_F^2}.
\end{equation*}
Finally, iterating the above will allow us to establish \eqref{eq:iterate-singular-vector}.
\end{proof}

From the theorem above, we see that, in the noisy case, both right and left singular vectors play some role: the {residual} term, $\widetilde{A}x_*-\widetilde{b}$, contributes to the horizon in the directions of the left singular vectors. Also, for $k$ large enough, the dominant direction is given by the right singular vector corresponding to the smallest singular value.

\begin{remark}
When the system \eqref{prb:inconsistent} is consistent {($E = 0$ and $\epsilon=0$)} and $x_0 \in {\rm range}(\widetilde{A}^\top)$, with the selection of $x_*=x_{\rm LS}$, the least squares solution to the system, Theorem \ref{thm:singular value additive noisy case} yields the following extension of Theorem~\ref{Steinerberger}, {allowing for low-rank matrices.}  
\end{remark}

\begin{corollary}\label{Steinerberger-extended} 
Let $\{x_k\}$ be the iterates of RK applied to a consistent system as in \eqref{prb:consistent}, where $\text{rank}(A) 
\leq n$. Assume that $x_0\in {\rm range}(A^\top)$. Let $x_{\rm LS}=A^\dagger b$ and let $v_j$ be the $j$th right singular vector of $A$ associated to the singular value $\sigma_j$. Then:
\begin{equation}
     \mathbb{E} \langle x_{k} - x_{\rm LS},v_j \rangle = \left(1-\frac{\sigma_j^2}{\|A\|_F^2}\right)^{k} \langle  x_0-x_{\rm LS},v_j \rangle,
\end{equation}
for $j = 1,..., \text{rank}(A)$.
\end{corollary}

We now show that Theorem~\ref{thm:singular value additive noisy case} indeed yields an estimate on the rate of convergence of the mean of RK iterates themselves (to a ball centered at $x_0^n + x_*^r$).

\begin{corollary}
\label{cor:better estimate}
Let $\{x_k\}$ be the iterates of RK applied to the doubly-noisy linear system~\eqref{prb:inconsistent}. 
Let $x_0$ be arbitrary. We have, for any $x_*\in \R^n$,
\begin{equation}\label{rate2}
 \| \mathbb{E}(x_{k}) -x_0^n- x_*^r\| \leq \left(1-\frac{1}{\widetilde{R}}\right)^{k} \| x_0^r-x_*^r \|+ \frac{\| \widetilde{A} x_* - \widetilde{b} \|}{\widetilde{\sigma}_{\rm min}}. 
\end{equation}  
where $\widetilde{R} = \|\widetilde A\|^2_F/\widetilde{\sigma}_{\rm min}^2$.
\end{corollary}
\begin{proof}
Let $c \in \mathbb{R}^{\rho}$ be a unit norm vector. Multiplying both sides of \eqref{eq:iterate-singular-vector} by $c_j$ and summing up for $j=1$ to $j=\rho={\rm rank}(\widetilde{A})$ we get 
\begin{align}\label{eq:along singular vectors}
 \mathbb{E} \langle  x_{k} - x_0^n-x_*^r,\sum_{j=1}^\rho c_j\widetilde v_j \rangle &=
 \sum_{j=1}^\rho \left(1-\frac{\widetilde{\sigma}_j^2}{\|\widetilde{A}\|_F^2}\right)^{k} \langle  x_0^r-x_*^r,c_j\widetilde v_j \rangle \nonumber \\
 &- \sum_{j=1}^\rho
 \left[1-\left(1-\frac{\widetilde{\sigma}_j^2}{\|\widetilde A\|_F^2}\right)^k\right]
 \frac{\langle \widetilde{A} x_* - \widetilde{b}, c_j \widetilde u_j \rangle} 
{\widetilde{\sigma}_{j}}.
\end{align}  
Note that $x_k-x_0^n-x_*^r\in {\rm range}(\widetilde A^\top)={\rm span}(\{\widetilde v_1,\widetilde v_2,...,\widetilde v_\rho\})$, and, without loss of generality, we can assume that $x_k-x_0^n-x_*^r\not =0$. So, we can choose $c_j$'s such that $\sum_{j=1}^\rho |c_j|^2=1$ and
\begin{equation*}
    \frac{\mathbb{E}(x_k)-x_0^n-x_*^r}{\|\mathbb{E}(x_k)-x_0^n-x_*^r\|}=\sum_{j=1}^\rho c_j \widetilde v_j.
\end{equation*}
Thus, with this set of $\{c_j\}$, the left side of \eqref{eq:along singular vectors} equals to $\|\mathbb{E}(x_k)-x_0^n-x_*^r\|$, and the right side can be bounded from above by 
\begin{align*}
  &  \left(1-\frac{\widetilde{\sigma}_\rho^2}{\|\widetilde{A}\|_F^2}\right)^{k} |\langle  x_0^r-x_*^r,\sum_{j=1}^\rho c_j \widetilde v_j \rangle |  + \left[1-\left(1-\frac{\widetilde{\sigma}_1^2}{\|\widetilde A\|_F^2}\right)^k\right]
 \frac{|\langle \widetilde{A} x_* - \widetilde{b}, \sum_{j=1}^\rho c_j \widetilde u_j \rangle |}  
{\widetilde{\sigma}_{\rho}},
\end{align*}
which, by applying the Cauchy-Schwarz inequality to the two inner products,
can be bounded from above further by
\begin{equation*}
     \left( 1-\frac{\widetilde{\sigma}_\rho^2}{\|\widetilde{A}\|_F^2}\right)^{k} \|x_0^r-x_*^r\|
 + 
 \left[1-\left(1-\frac{\widetilde{\sigma}_1^2}{\|\widetilde A\|_F^2}\right)^k\right]
 \frac{\|\widetilde{A} x_* - \widetilde{b}\|}   
{\widetilde{\sigma}_{\rho}}.
\end{equation*}
This implies the inequality stated in the corollary.
\end{proof}

\begin{remark}\label{remark:after-error_of_the_mean}
It is interesting to compare the estimates of the error in Theorems \ref{thm:bergou} and \ref{thm:DoublyNoisy_x} with that in Corollary \ref{cor:better estimate}. The former measures the mean squared error, $\mathbb{E}\|x_k-x_{\rm LS}\|^2$ or $\mathbb{E}\| x_{k} -x_0^n- x_*^r\|^2$, while the latter bounds the error of the mean, $\| \mathbb{E}(x_{k}) -x_0^n- x_*^r\|$.
Using {Jensen's} inequality, we know that 
$\| \mathbb{E}(x_{k}) -x_0^n- x_*^r\|\leq \left(\mathbb{E}\| x_{k} -x_0^n- x_*^r\|^2\right)^{1/2}$. 
{When $\|x_0^r-x_*^r\|\leq 2 \mfrac{\|\widetilde{A} x_* - \widetilde{b}\|}   
{\widetilde{\sigma}_{min}}$, square root of the bound of Theorem \ref{thm:DoublyNoisy_x} is smaller than the bound of  Corollary \ref{cor:better estimate} while when $\|x_0^r-x_*^r\| \left(1-\left(1-\mfrac{\widetilde{\sigma}_{\rm min}^2}{\|\widetilde A\|_F^2}\right)^k \right) \geq 2 \mfrac{\|\widetilde{A} x_* - \widetilde{b}\|}   
{\widetilde{\sigma}_{min}}$, the bound of  Corollary \ref{cor:better estimate} is smaller than the square root of the bound of Theorem \ref{thm:DoublyNoisy_x}. For $k$ large, both bounds have the same convergence horizon.}
In Section~\ref{sec:numerics},  we will numerically compare these bounds.
\end{remark}

\section{Bounding the limit points of the RK iterates} \label{sec:bounding_balls}
From Theorems~\ref{thm:bergou} and  \ref{thm:DoublyNoisy_x} and Corollary~\ref{thm:singular value additive noisy case}, we can see that the limit points of the RK iterates, $\{x_k\}$, are contained in certain balls in $\R^n$. The smaller the balls, the sharper the estimates of the location of these limit points. Two families of balls can be deduced from these theorems: the one centered at $x_{\rm LS}$ under the assumption that $x_0-x_{\rm LS}\in {\rm range}(\widetilde{A}^\top)$ as given in Theorem~\ref{thm:bergou} and the one without using this assumption as given in Theorem~\ref{thm:DoublyNoisy_x} and Corollary~\ref{thm:singular value additive noisy case}. These two families of balls overlap but use different choices for their centers.
In this section, we search for the smallest balls in each of these two families. As it turns out, the two families of balls share the same smallest ball.

\subsection{With the assumption $x_0-x_{\rm LS}\in {\rm range}(\widetilde{A}^\top)$} We start with Theorem \ref{thm:bergou}, which asserts that, under the assumption that $x_0-x_{\rm LS}\in {\rm range}(\widetilde{A}^\top)$, the sequence of the iterates of RK applied to the doubly noisy linear system \eqref{prb:inconsistent}
approaches, in expectation, to a ball in $\R^n$ centered around the least squares solution $x_{\rm LS}$ of an associated underlying consistent system $Ax=b$ and with radius $r_{A,b}:=\mfrac{\| E x_{\rm LS} - \epsilon \|}{\sigma_{\rm min}(\widetilde{A})},$ where $E=\widetilde A - A$ and $\epsilon=\widetilde b-b.$ We denote this ball by $B(x_{\rm LS},r_{A,b})$. Let $x_0$ be given and consider the set ${\cal K}(x_0)$ of all possible pairs $(A,b)$ such that the system $Ax=b$ is consistent and $x_0-x_{\rm LS}\in {\rm range}(\widetilde A^\top)$. We can see that the sequence of RK iterates, {$\{x_k\}$}, approaches the intersection of the corresponding balls associated with the pairs $(A,b)\in {\cal K}(x_0)$. More precisely, if $L(x_0)$ denotes the collection of all the limit points of $\{x_k\}$ with $x_0$ as the initial term, then
\begin{equation}   L(x_0)\subseteq\bigcap_{(A,b)\in {\cal K}(x_0)} 
    B(x_{\rm LS},r_{A,b}).
\end{equation}
In the following result, we identify the balls with the smallest radius that attract the limit points of the sequence of RK iterates and determine all such balls.

\begin{theorem}\label{thm:smallest-ball}
Fix $\widetilde{A} \in \mathbb{R}^{m \times n}$ and $\widetilde{b} \in \mathbb{R}^m$. Let $(\widehat A,\widehat b) \in \mathbb{R}^{m \times n} \times \mathbb{R}^{m}$ be a minimizer of $\|Ex_{\rm LS}-\epsilon\|$ among all pairs $(A,b)$ where $E=\widetilde A - A$, $\epsilon=\widetilde b-b$, $x_{\rm LS}=A^\dagger b$, subject to the constraint that $Ax=b$ is consistent. That is, 
    \begin{equation}\label{constrained min}
        (\widehat A,\widehat b) \in \arg\min_{\substack {(A ,b )  \in \R^{m\times n} \times \R^{m}\\ Ax=b \\ \text{is consistent}}}  \|Ex_{\rm LS}-\epsilon\|.
    \end{equation}
    Denote $\widehat E=\widetilde A - \widehat A$, $\widehat \epsilon=\widetilde b-\widehat b$, and $\widehat x_{\rm LS}=\widehat A^\dagger \widehat b$. Then,
         \begin{equation}\label{(i)}
         \|\widehat E \widehat x_{\rm LS}- \widehat \epsilon\| = \|\widetilde A \widetilde x_{\rm LS}- \widetilde b\| ~{\it where}~ \widetilde x_{\rm LS}=\widetilde A^\dagger \widetilde b,\end{equation}
    and \begin{equation}\label{(ii)} \widehat x_{\rm LS} = \widetilde x_{\rm LS} +y ~{\it with}~y\in {\rm null}(\widetilde{A}).\end{equation} 
\end{theorem}

\begin{proof}
We first demonstrate that, indeed, the constrained minimization problem in $(A,b)$ does have a solution by deriving a solution to \eqref{constrained min}. We do this by reparameterizing the consistency constraint as follows.

Let $Ax=b$ be any consistent system then $\widetilde{A}=A+E$ and $\widetilde{b}=b+\epsilon$.
Consider the compact singular value decomposition $A=U\Sigma V^\top$ where $U \in \R^{m\times r}$ and $V \in \R^{n\times r}$ have orthonormal columns, and $\Sigma={\rm diag}(\sigma_1,...,\sigma_r)$ where $r={\rm rank}(A)$ and $\sigma_1\geq \sigma_2\geq \cdots \geq \sigma_r>0$.
Since $Ax=b$ is consistent, there exists an $x_b \in \R^n$ such that $b=Ax_b=U\Sigma V^\top x_b$.
Hence, the least squares solution of $Ax=b$ can be expressed as $x_{\rm LS}=A^\dagger b=V\Sigma^{-1}U^\top b=VV^\top x_b$ and we have
\begin{equation} \label{eq1}
    Ex_{\rm LS}-\epsilon=(\widetilde A -A)x_{\rm LS}-(\widetilde b -b)= \widetilde A x_{\rm LS} - \widetilde b = \widetilde A VV^\top x_b - \widetilde b.
\end{equation}
Therefore, the following two minimization problems are equivalent 
\begin{equation}
    \min_{\substack {(A ,b )  \in \R^{m\times n} \times \R^{m}\\ Ax=b \\ \text{is consistent}}}  \|Ex_{\rm LS} - \epsilon\| =  \min_{\substack {(V, x_b) \in \R^{n\times r} \times \R^{n} \\ V^\top V = I_r \\ r\leq n}}  \|\widetilde A VV^\top x_b- \widetilde b\|, \label{eq3} 
\end{equation}
and their solution sets are related through the transformations between the two feasible sets induced by the following mapping from $(V,x_b)$ to $(A,b)$:
\begin{equation}\label{mapping}
A=U\Sigma V^\top ~{\rm and}~b=U\Sigma V^\top x_b.
\end{equation}
Note that vectors expressed in the form $VV^\top x_b$ would exhaust all $\R^n$ with all possible choices of $(V \in \R^{n\times r},x_b \in \R^{n})$ such that  $V^\top V = I_r$ and $r\leq n$. Thus, we have
\begin{align}
 \min_{\substack {(V \in \R^{n\times r},x_b \in \R^{n}) \\ V^\top V = I_r \\ r\leq n}}  \|\widetilde A VV^\top x_b-\widetilde b\|
    = \min_{\substack {x \in \R^{n}}}  \|\widetilde A x -\widetilde b\| 
    = \|\widetilde A \widetilde x_{\rm LS} -\widetilde b\| \label{eq33}
\end{align}
\text{with} $\widetilde x_{\rm LS}=\widetilde A^\dagger \widetilde b$.

From equalities \eqref{eq33}, we see that we can take $(\widehat{\widehat{A}},\widehat{\widehat{b}})=(U\Sigma V^\top, \allowbreak U\Sigma V^\top x_b)$ with $V$ and $x_b$ satisfying $V^\top V=I_r$ and $VV^\top x_b=\widetilde{x}_{\rm LS}$, $U$ arbitrary satisfying $U^\top U=I_r$, and $\Sigma$ any diagonal matrix in $\R^{r\times r}$ with positive entries along the diagonal. Then $(\widehat{\widehat{A}},\widehat{\widehat{b}})$ would solve the minimization problem \eqref{constrained min}. There are many possible solutions, given the many choices of such $U$, $V$, $\Sigma$, and $x_b$. 

Next, we show that every solution to \eqref{constrained min} can be given in the form of $(\widehat{\widehat{A}},\widehat{\widehat{b}})$.
Let 
\begin{equation*}
    (\widehat A,\widehat b) \in \arg \min_{\substack {(A ,b )  \in \R^{m\times n} \times \R^{m}\\ Ax=b \\ \text{is consistent}}}  \|Ex_{\rm LS}-\epsilon\|, 
\end{equation*}
 with $\widehat A = U\Sigma \widehat V^\top$ (a compact SVD of $\widehat{A}$) and $\widehat b = \widehat A \widehat x_b$ where $\widehat V \in \R^{n\times r}$ has orthonormal columns. Then, from \eqref{eq1} and \eqref{eq3}, we see that 
 \begin{align*}
    (\widehat V, \widehat x_b ) \in \arg \min_{\substack {(V \in \R^{n\times r},x_b \in \R^{n}) \\ V^\top V \!=\! I_r \\ r\leq n}}  \|\widetilde A VV^\top x_b-\widetilde b\| 
  = \{(V,x_b)\!:\!VV^\top x_b\in\arg \min_{x\in \R^n}\|\widetilde{A}{x}-\widetilde{b}\|\},
 \end{align*}
 and, additionally, given \eqref{eq33},
  $ \|\widehat{E}\widehat{x}_{\rm LS}-\widehat{\epsilon}\| = \|\widetilde{A}\widetilde{x}_{\rm LS}-\widetilde{b}\|.$ 
 This verifies \eqref{(i)}.
 Furthermore, it follows that $\widehat V \widehat V^\top \widehat{x}_b = \widetilde x_{\rm LS}+y$ with $y\in {\rm null}(\widetilde{A})$. 
 So, when $(\widehat{A},\widehat{b})$ solves the constrained minimization problem, one must have
 $ \widehat{A}=U\Sigma \widehat{V}^\top~{\rm and}~\widehat{b}=U\Sigma \widehat{V}^\top \widehat{x}_b$
with $\widehat{x}_{\rm LS}=\widehat{V}\widehat{V}^\top \widehat{x}_b=\widetilde{x}_{\rm LS}+y$ for any $y\in {\rm null}(\widetilde{A})$. This is \eqref{(ii)}, completing the proof.
\end{proof}

With this theorem, we can immediately obtain the following description of the balls with the minimum radius.
\begin{corollary}\label{cor:center1}
The minimum radius among all balls $B(x_{\rm LS},r_{A,b})$ generated by all {tuples} ${(A,b)}\in {\cal K}{(x_0)}$  is given by 
\begin{align*}
 & \min\{ r_{A,b}~:~Ax =b ~{\it is~ consistent},~ {\it and}  ~x_0 - x_{\rm LS}\in {\rm range}(\widetilde{A}^\top) \} =\frac{\|\widetilde{A}\widetilde{x}_{\rm LS}-\widetilde{b}\|}{\widetilde \sigma_{\rm min}}.   
\end{align*}
Moreover, the centers of the balls with the minimum radius depend on where $x_0$ is located and can be given uniquely by
\begin{equation*}
    x_{\rm LS}=\widetilde{x}_{\rm LS}+x_0^n.
\end{equation*}
\end{corollary}

\subsection{Without any assumption on $x_0$} Next, we use Theorem~\ref{thm:DoublyNoisy_x} and Corollary~\ref{cor:better estimate}, which removes the use of the underlying consistent systems and their least squares solution $x_{\rm LS}$. So, the center is $x_0^n+x_*^r$ instead of $x_{\rm LS}$ and the consistent system is taken to be $\widetilde{A}y=\widetilde{A}x_*$ for unknown vector $y$. Thus, we see that the limit points of $\{x_k\}$ are contained in the intersection of all the balls given by
\begin{equation*}
B\left(x_0^n+x_*^r,\frac{\|\widetilde{A}x_*-\widetilde{b}\|}{\widetilde{\sigma}_{min}}\right),
\end{equation*}  
where $x_0$ and $x_*$ can run through all vectors in $\R^n$.

Now, it is easy to see that the residual of the least squares solution must give the smallest radius. More precisely, we have the following.

\begin{figure}
    \centering
    \includegraphics[width=0.45\textwidth]{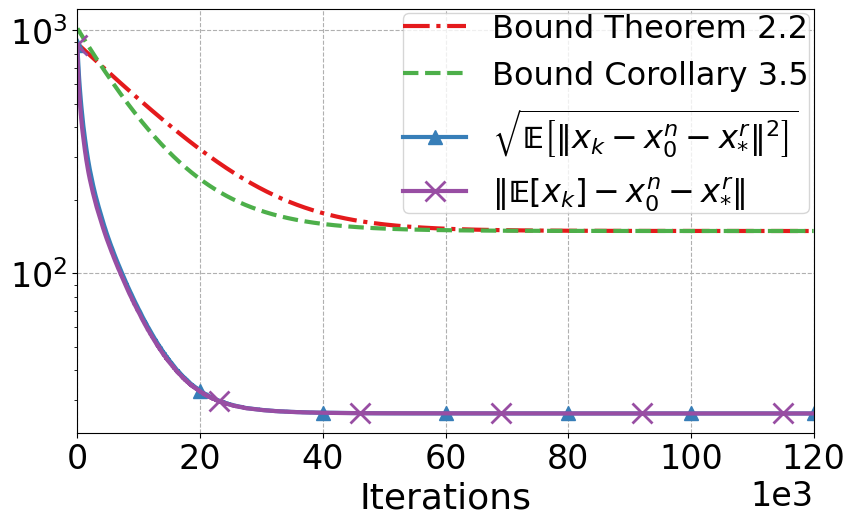}
    \includegraphics[width=0.45\textwidth]{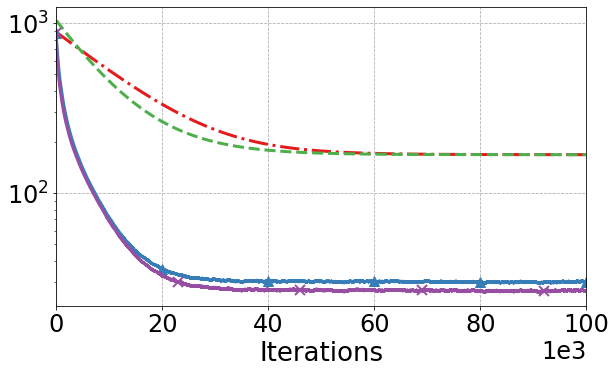}
 \caption{{The approximation errors  $\sqrt{\Exp{\|x_k-x_0^n-x_*^r\|^2}}$ and $\|\Exp{x_k}-x_0^n-x_*^r\|$} of RK applied to $\widetilde Ax \approx \widetilde b$, square root of bound of Theorem \ref{thm:DoublyNoisy_x}, and bound of Corollary \ref{cor:better estimate}. We have $m=1000$, $n=500$, and ${\rm rank}(\widetilde A)=300$. {The initialization }$x_0$ and {reference point} $x_*$ are random. We have $\|\widetilde b_{{\rm Col}(\widetilde A)^\perp}\|=\beta$. On the left, $\beta=10$. Here, on the right, $\beta=10000$.}
  \label{fig:comp_bounds3_noisy}
\end{figure}

\begin{corollary}\label{thm:smallest-ball2}
Among all the balls
\begin{equation*}
 B\left(x_0^n+x_*^r,\frac{\|\widetilde{A}x_*-\widetilde{b}\|}{\widetilde{\sigma}_{min}}\right),~x_*\in \R^n,   
\end{equation*}
the smallest radius is attained when $x_*=\widetilde{x}_{\rm LS}+y$, where $\widetilde{x}_{\rm LS}$ is the least squares solution to \eqref{prb:inconsistent} and $y\in {\rm null}(\widetilde{A})$. Moreover,  the center of the smallest ball is given by $x_0^n+\widetilde{x}_{\rm LS}$.   
\end{corollary}

\begin{proof}
    The first part of the statement follows directly from minimizing the radius, and the second part follows from the fact that $\widetilde{x}_{\rm LS}^r=\widetilde{x}_{\rm LS}$. 
\end{proof}

\begin{figure}
    \centering
    \includegraphics[width=0.45\textwidth]{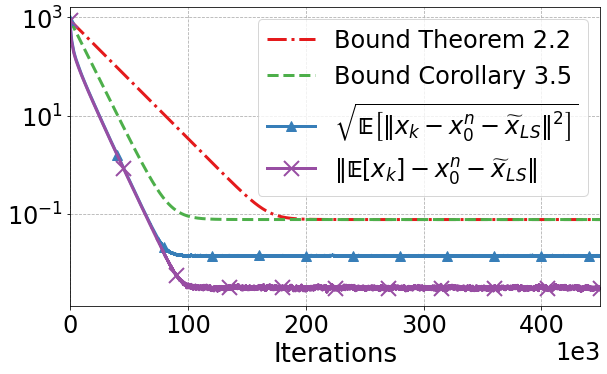}
    \includegraphics[width=0.45\textwidth]{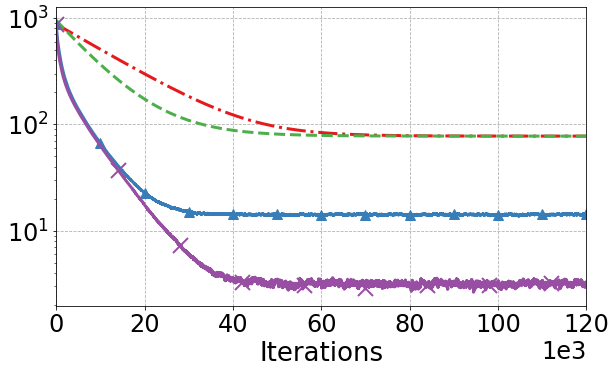}
 \caption{{The approximation errors  $\sqrt{\Exp{\|x_k-x_0^n-\widetilde x_{\rm LS}\|^2}}$ and $\|\Exp{x_k}-x_0^n- \widetilde x_{\rm LS}\|$} of RK applied to $\widetilde Ax \approx \widetilde b$, square root of bound of Theorem \ref{thm:DoublyNoisy_x}, and bound of Corollary \ref{cor:better estimate}. We have $m=1000$, $n=500$, and ${\rm rank}(\widetilde A)=300$. The initialization $x_0$ is random and $\widetilde x_{\rm LS}=\widetilde A^\dagger\widetilde b$. We have $\|\widetilde b_{{\rm Col}(\widetilde A)^\perp}\|=\beta$. On the left, $\beta=10$. On the right, $\beta=10000$.}
  \label{fig:comp_bounds3_noisy_xls}
\end{figure}
\begin{figure}
    \centering
    \includegraphics[width=0.45\textwidth]{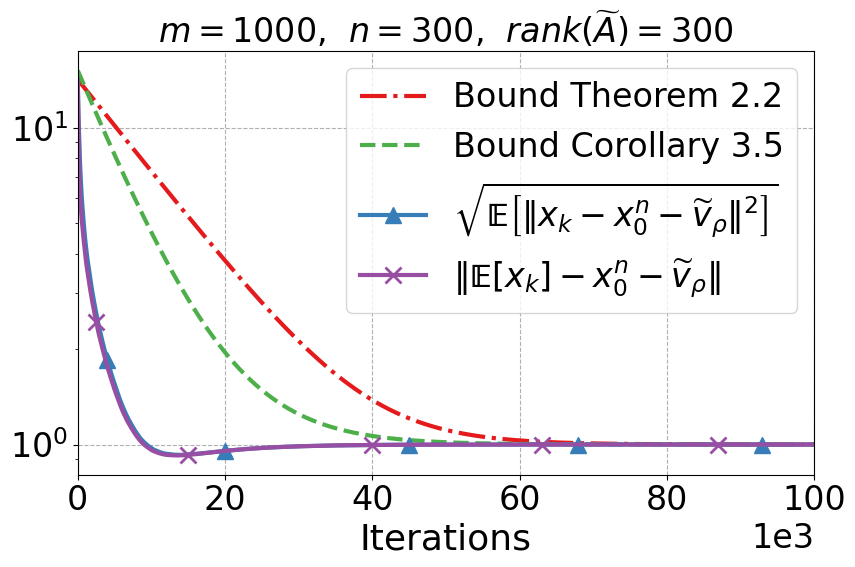}
    \includegraphics[width=0.45\textwidth]{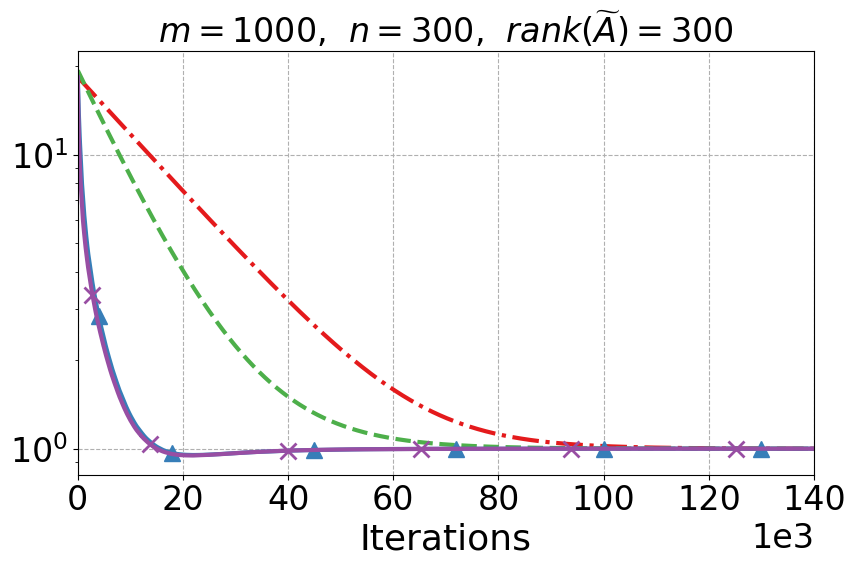}
 \caption{{The approximation errors  $\sqrt{\Exp{\|x_k-x_0^n-\widetilde{v}_{\rho}\|^2}}$ and $\|\Exp{x_k}-x_0^n- \widetilde{v}_{\rho}\|$} of RK applied to $\widetilde Ax=0$ $(\widetilde b=0)$ with ${\rm rank}(\widetilde A)=\rho$, square root of bound of Theorem \ref{thm:DoublyNoisy_x}, and bound of Corollary \ref{cor:better estimate}. The initialization $x_0$ is arbitrary. On the left, $\widetilde A$ is of low rank. On the right, $\widetilde A$ is full-rank. }
  \label{fig:comp_bounds4_smallestSingularVector}
\end{figure}

\section{Numerical results} \label{sec:numerics}

In this section, we present numerical results that support our theoretical findings. 
First, we compare our general bounds of Theorem \ref{thm:DoublyNoisy_x} and Corollary \ref{cor:better estimate}, which accommodate arbitrary starting points and general reference points. We compare these bounds on noisy inconsistent systems, 
generated using synthetic data, and on real-world data from the LIBSVM~\cite{LIBSVM} dataset. {In our theoretical results, Theorem \ref{thm:DoublyNoisy_x} bounds $\Exp{\|x_k-x_0^n-x_*^r\|^2}$ while Corollary \ref{cor:better estimate} bounds $\|\Exp{x_k}-x_0^n-x_*^r\|$. The former describes a ball that attracts the RK iterates in the mean-squared sense, while the latter describes a ball that attracts the mean of the iterates. In the numerical results, we compare the square root of the bound of Theorem \ref{thm:DoublyNoisy_x} and the bound of Corollary \ref{cor:better estimate} alongside $\sqrt{\Exp{\|x_k-x_0^n-x_*^r\|^2}}$ and $\|\Exp{x_k}-x_0^n-x_*^r\|$ where the expectation is computed by averaging over several runs of the algorithm.}
Second, we empirically validate the equality of Theorem \ref{thm:singular value additive noisy case} that describes the convergence along the singular vectors of the noisy matrix in doubly noisy linear systems.
Finally, we illustrate the limiting balls containing the final iterates of RK through simple examples in the $2$D case. We demonstrate that the ball centered at $x_0^n+\widetilde x_{\rm LS}$ achieves the smallest radius, confirming our theoretical analysis.
\begin{figure}
    \centering
    \includegraphics[width=0.44\textwidth]{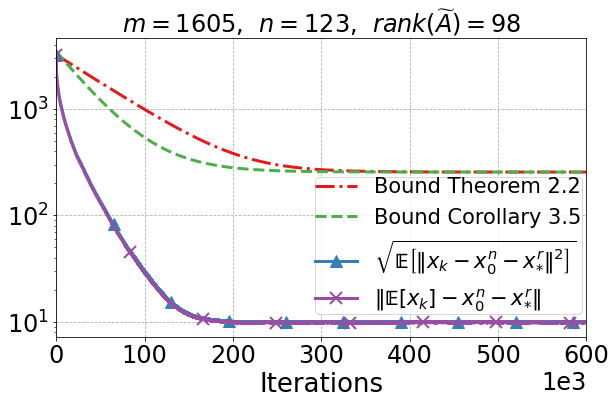}
    \includegraphics[width=0.44\textwidth]{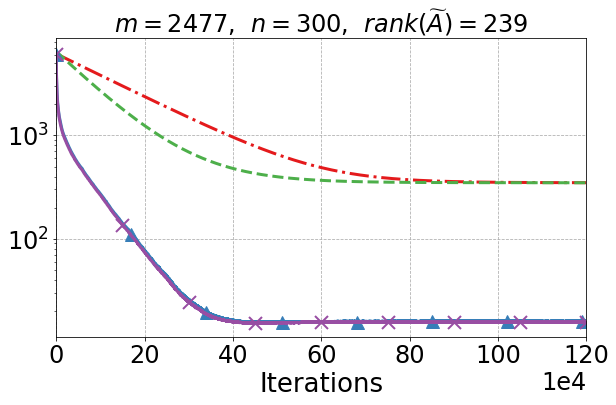} \\ 
    \includegraphics[width=0.44\textwidth]{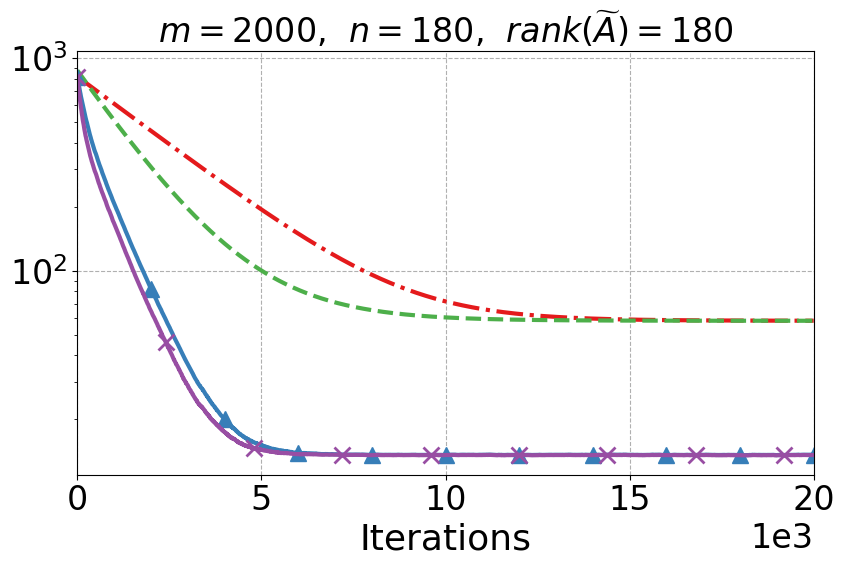}
    \includegraphics[width=0.44\textwidth]{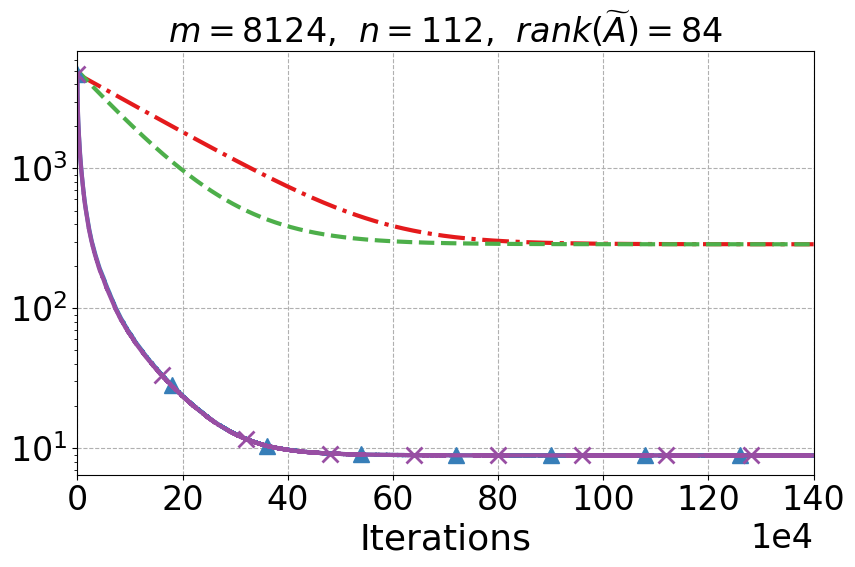}
    \caption{The approximation errors  $\sqrt{\Exp{\|x_k-x_0^n-x_*^r\|^2}}$ and $\|\Exp{x_k}-x_0^n-x_*^r\|$, square root of bound of Theorem \ref{thm:DoublyNoisy_x}, and bound of Corollary \ref{cor:better estimate}. The initial point, $x_0$, and the reference point, $x_*$, are random. First row, left to right: \texttt{a1a}, \texttt{w1a}; second row, left to right:  \texttt{dna}, and \texttt{mushrooms} datasets from LIBSVM \cite{LIBSVM}.} 
  \label{fig:comp_bounds_general}
\end{figure}

In all the figures, the approximation errors are given by averaging over $20$ runs of RK. We use random starting points that ensure that the initial value of the theoretical bounds is much larger than the convergence horizon, so that we can witness both the convergence and horizon behaviors of RK. All codes for the experiments are available at: 
\href{https://github.com/SoumiaBouch/Where-Have-All-the-Kaczmarz-Iterates-Gone}{https://github.com/SoumiaBouch/Where-Have-All-the-Kaczmarz-Iterates-Gone.}

\myNum{i}\smartparagraph{Comparing bounds for noisy linear systems.}
To compare our theoretical bounds in the noisy case, we generate the noisy linear system as follows: $\widetilde{A}=UV$ where $U$ and $V$ are of size $(m,r)$ and $(r,n)$ respectively with i.i.d. Gaussian entries. We set $\widetilde b = y+\beta w$, where $y$ is a random vector from the column space of $\widetilde A$ and $w$ is a random vector of the unit norm from the orthogonal complement of the column space of $\widetilde A$. {In our setup, $\beta>0$ is a scalar controlling the distance of $\widetilde{b}$ from the column space of $\widetilde{A}$ and} we have $\|\widetilde b_{{\rm Col}(\widetilde A)^\perp}\|=\beta$. 

Figures \ref{fig:comp_bounds3_noisy} and \ref{fig:comp_bounds3_noisy_xls} show the approximation errors of RK applied to $\widetilde A x = \widetilde b$ with different values of $\beta$, the square root of bound of Theorem \ref{thm:DoublyNoisy_x} and the bound of Corollary \ref{cor:better estimate}. We do not require the starting point to be in the row space of the noisy matrix for this experiment. Figure \ref{fig:comp_bounds3_noisy} compares the bounds when considering a randomly selected point $x_*$ as the reference point. At the same time, Figure \ref{fig:comp_bounds3_noisy_xls} shows the results when considering $\widetilde x_{\rm LS}$ as a reference point, where $\widetilde x_{\rm LS}=\widetilde{A}^\dagger\widetilde{b}$. The results show that the compared bounds are valid and have the same convergence horizon. 
{$\|\Exp{x_k}-x_0^n-x_*^r\| \leq \sqrt{\Exp{\|x_k-x_0^n-x_*^r\|^2}}$ and in some cases the equality is achieved, which validates Remark \ref{remark:after-error_of_the_mean}. For the tested values, the bound of Corollary \ref{cor:better estimate} exhibits a faster convergence rate than the square root bound in Theorem \ref{thm:DoublyNoisy_x}. Although the difference between the bounds can be large, the difference between the approximation errors is tight.}

Figure~\ref{fig:comp_bounds4_smallestSingularVector} shows the approximation errors for RK applied to the homogeneous system $\widetilde{A}x=0$ and the bounds when the reference point $x_*$ is equal to the smallest right singular vector of $\widetilde{A}$. In this case, both bounds are tight, with the bound of Corollary~\ref{cor:better estimate} showing a better convergence rate than the square root of the bound of Theorem~\ref{thm:DoublyNoisy_x}.

\begin{figure}
    \centering
    \includegraphics[width=0.45\textwidth]{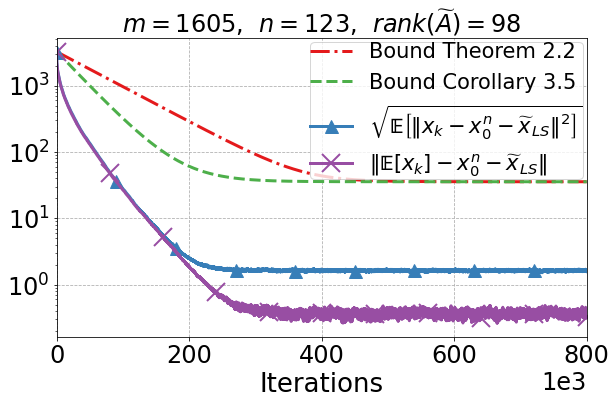}
    \includegraphics[width=0.45\textwidth]{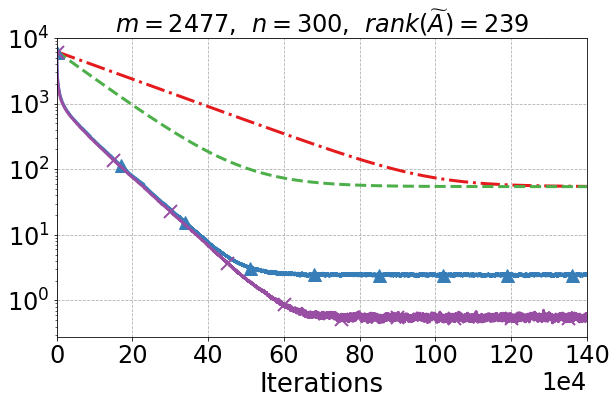}
    \includegraphics[width=0.45\textwidth]{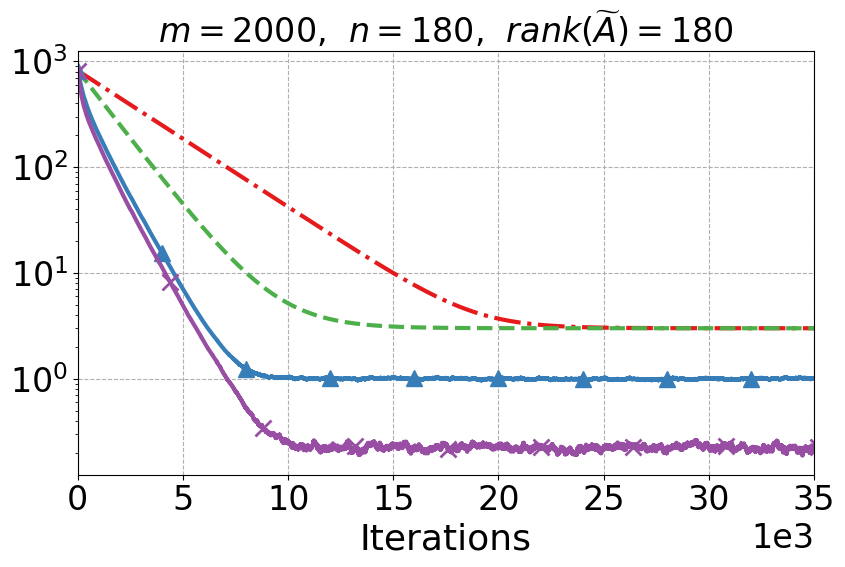}
    \includegraphics[width=0.45\textwidth]{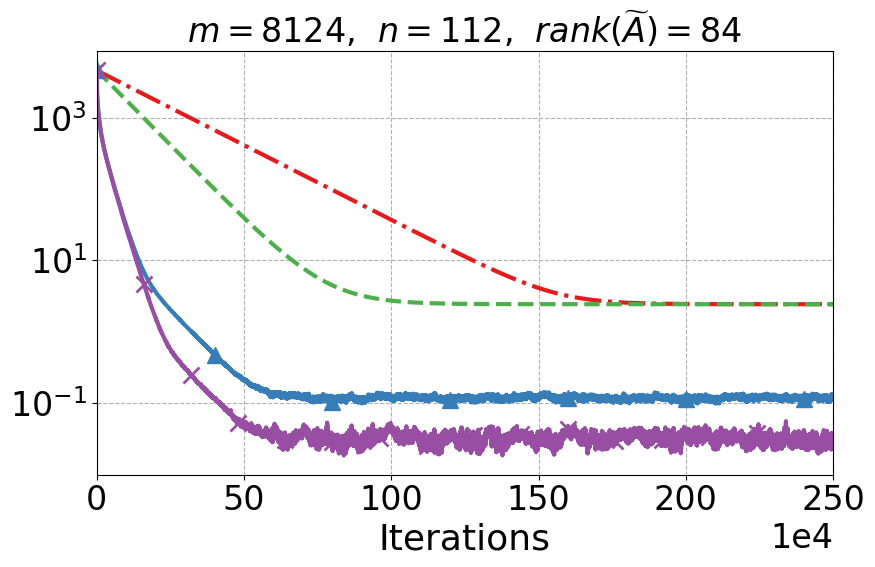}
    \caption{The approximation errors  $\sqrt{\Exp{\|x_k-x_0^n-\widetilde x_{\rm LS}\|^2}}$ and $\|\Exp{x_k}-x_0^n- \widetilde x_{\rm LS}\|$, square root of bound of Theorem \ref{thm:DoublyNoisy_x}, and bound of Corollary \ref{cor:better estimate}. The initial point, $x_0$ is random and $\widetilde{x}_{\rm LS}=\widetilde{A}^\dagger \widetilde{b}$. {First row, left to right: \texttt{a1a}, \texttt{w1a}; second row, left to right:  \texttt{dna}, and \texttt{mushrooms}} datasets from LIBSVM \cite{LIBSVM}.} 
  \label{fig:comp_bounds_xls_tld}
\end{figure}

\myNum{ii}\smartparagraph{Comparing bounds {on real-world data.}}
In these experiments, we compare our bounds of Theorem \ref{thm:DoublyNoisy_x} and Corollary \ref{cor:better estimate} using real-world data $\widetilde A$ and $\widetilde b$ from LIBSVM datasets \cite{LIBSVM}. We consider this data noisy as noise is inevitable in practice. Moreover, the resulting system $\widetilde A x = \widetilde b$ is inconsistent. 

Figures \ref{fig:comp_bounds_general} and \ref{fig:comp_bounds_xls_tld} show the approximation errors of RK applied to $\widetilde A x = \widetilde b$, the square root of the bound of Theorem \ref{thm:DoublyNoisy_x}, and the bound of Corollary \ref{cor:better estimate}. We do not require the starting point to be in the row space of the noisy matrix. 
Figure \ref{fig:comp_bounds_general} shows the results when choosing a random point $x_*$ as a reference point, while Figure \ref{fig:comp_bounds_xls_tld} shows the results when selecting $\widetilde x_{\rm LS}$, the least squares solution of the noisy system, as a reference point. 
Both our derived bounds are valid in this realistic setting, with the bound of Corollary \ref{cor:better estimate} showing a faster convergence rate than the square root bound in Theorem \ref{thm:DoublyNoisy_x}. 
Moreover, choosing $\widetilde{x}_{\rm LS}$ as a reference point yields a smaller convergence horizon compared to a general point $x_*$.

\myNum{iii}\smartparagraph{Empirical validation of Theorem \ref{thm:singular value additive noisy case}.} Figures \ref{fig:Thm3-2_general} and \ref{fig:Thm3-2_xls} show the values of the quantities of the equality of Theorem \ref{thm:singular value additive noisy case} across iterations and for some selected right singular vectors when $x_*$ is arbitrary and when $x_*=\widetilde x_{\rm LS}$. The quantity $(a)= \langle x_{k} - x_0^n-x_*^r,\widetilde v_j \rangle$ is on average empirically equal to $(b)-(c)$  as proven theoretically, where $(b)=\left(1-\mfrac{\widetilde{\sigma}_j^2}{\|\widetilde{A}\|_F^2}\right)^{k}  \langle  x_0^r-x_*^r,\widetilde v_j \rangle$, and $(c)=\left[1- \left(1-\mfrac{\widetilde{\sigma}_j^2}{\|\widetilde A\|_F^2}\right)^k\right] 
 \mfrac{\langle \widetilde{A} x_* - \widetilde{b}, \widetilde u_j \rangle} 
{\widetilde{\sigma}_{j}}$.

\myNum{iv}\smartparagraph{Limiting ball of final RK iterates.}
Figure \ref{fig:circles1} shows the path of RK iterates and randomly selected circles described in Corollaries \ref{cor:center1} and \ref{thm:smallest-ball2}. We test both cases, when the condition $x_0 - x_{\rm LS} \in {\rm range}(\widetilde A ^\top)$ is satisfied and when $x_0$ is arbitrary.
As the result shows, the limiting circles bound the final iterates of RK. 
Additionally, in both cases, the circle centered around $x_0^n+\widetilde x_{\rm LS}$, where $\widetilde x_{\rm LS}=\widetilde A^\dagger \widetilde b$, and of radius $\|\widetilde A \widetilde x_{\rm LS} - \widetilde b \|/ \widetilde \sigma_{\rm min}$ is the one with the smallest radius among the tested values, which validates our theoretical findings stated in Section \ref{sec:bounding_balls}.

\begin{figure}
    \centering
    \includegraphics[width=0.32\textwidth]{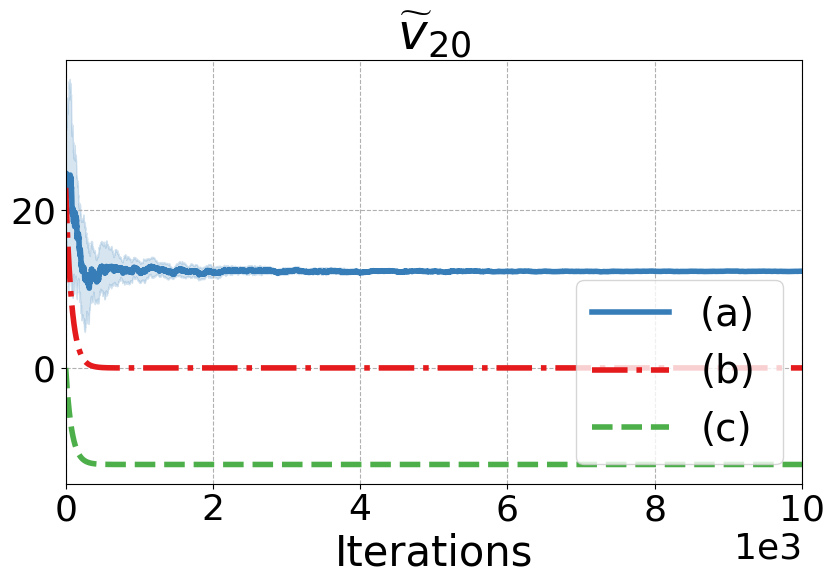}
    \includegraphics[width=0.32\textwidth]{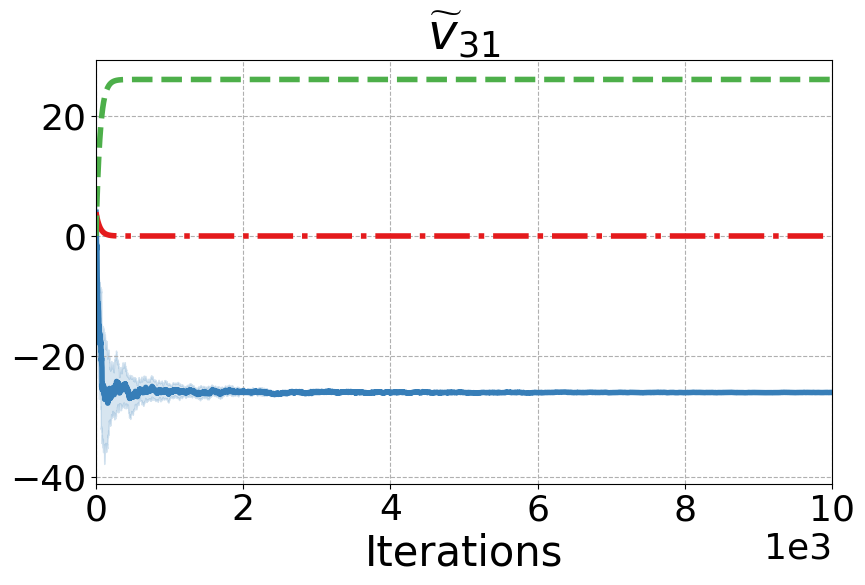}
    \includegraphics[width=0.32\textwidth]{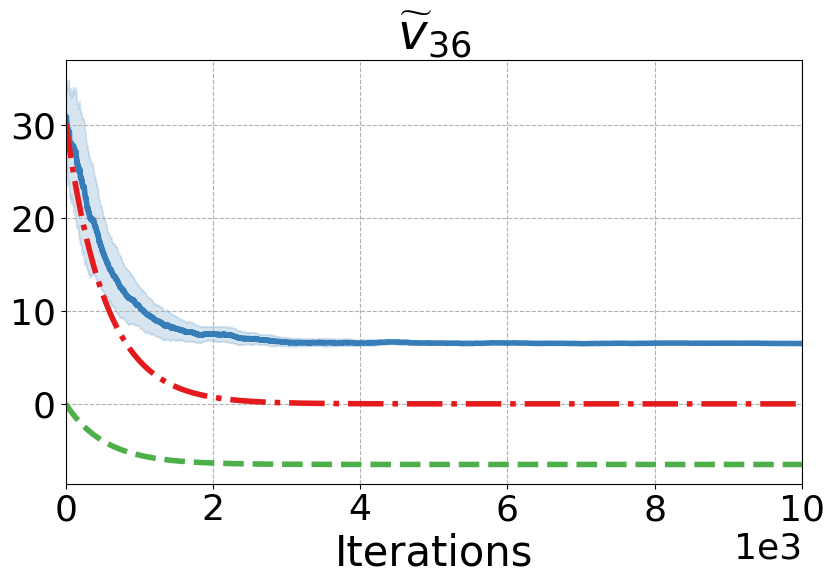}
    \includegraphics[width=0.32\textwidth]{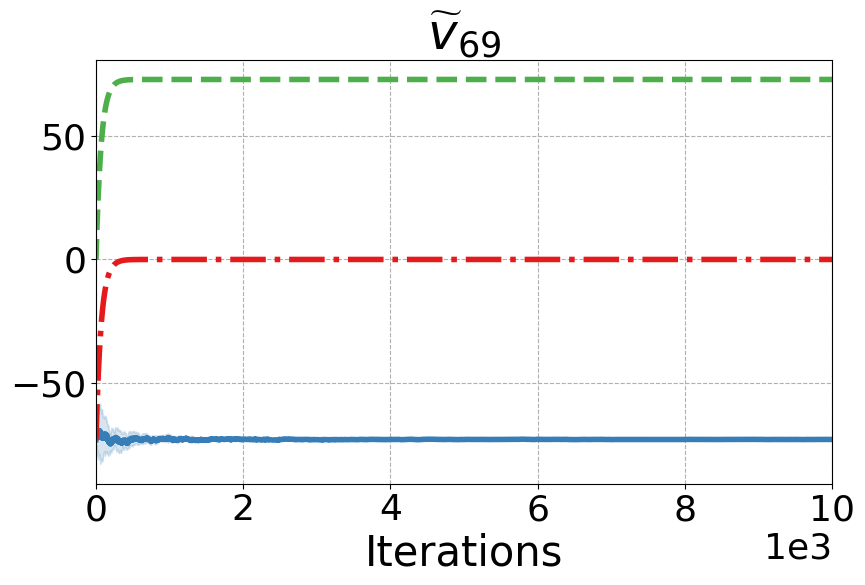}
    \includegraphics[width=0.32\textwidth]{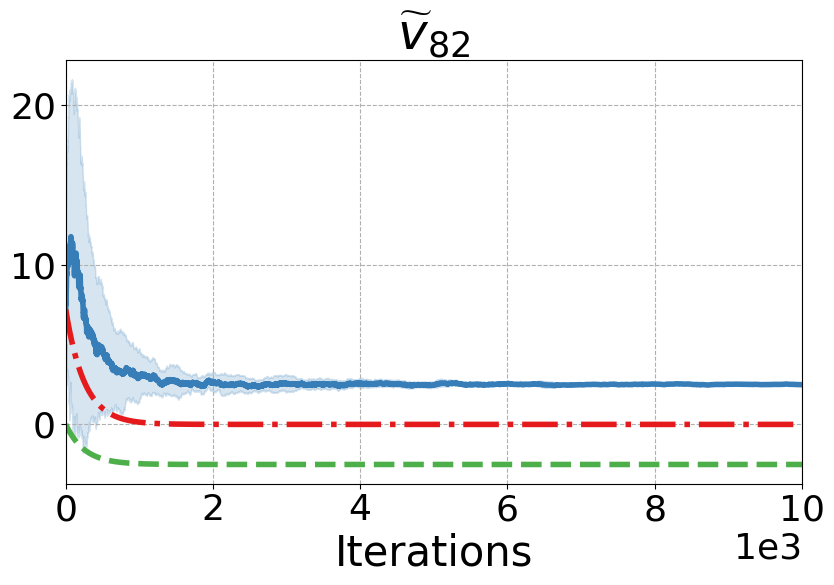}
    \includegraphics[width=0.32\textwidth]{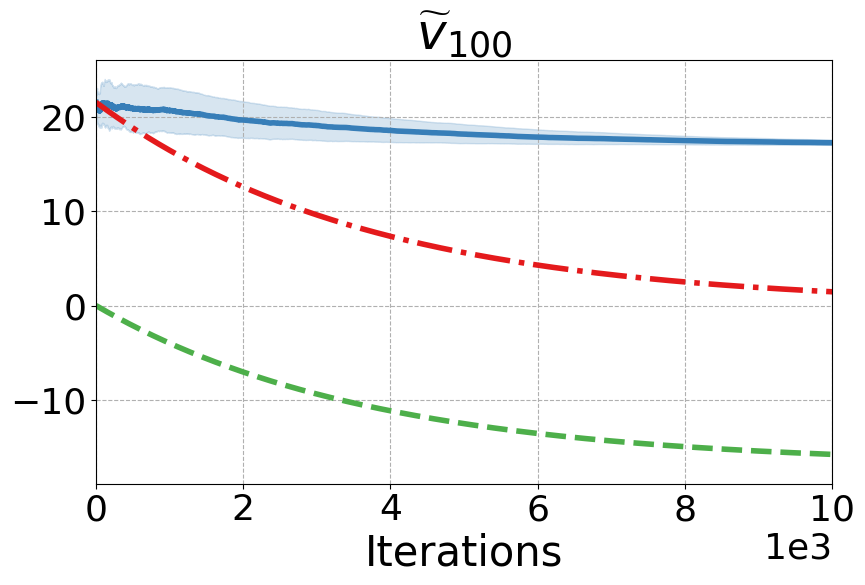}  
 \caption{Quantities of the equality of Theorem \ref{thm:singular value additive noisy case} for multiple right singular vectors. $(a)= \langle x_{k} - x_0^n-x_*^r,\widetilde v_j \rangle$ (averaged over $20$ runs), $(b)=\left(1-\frac{\widetilde{\sigma}_j^2}{\|\widetilde{A}\|_F^2}\right)^{k}  \langle  x_0^r-x_*^r,\widetilde v_j \rangle$, and $(c)=\left[1- \left(1-\mfrac{\widetilde{\sigma}_j^2}{\|\widetilde A\|_F^2}\right)^k\right] 
 \mfrac{\langle \widetilde{A} x_* - \widetilde{b}, \widetilde u_j \rangle} 
{\widetilde{\sigma}_{j}}$. We have $m=1000$, $n=200$, ${\rm rank}(\widetilde A)=100$, $x_0$ and $x_*$ are arbitrary.}
  \label{fig:Thm3-2_general}
\end{figure}

\begin{figure}
    \centering
    \includegraphics[width=0.32\textwidth]{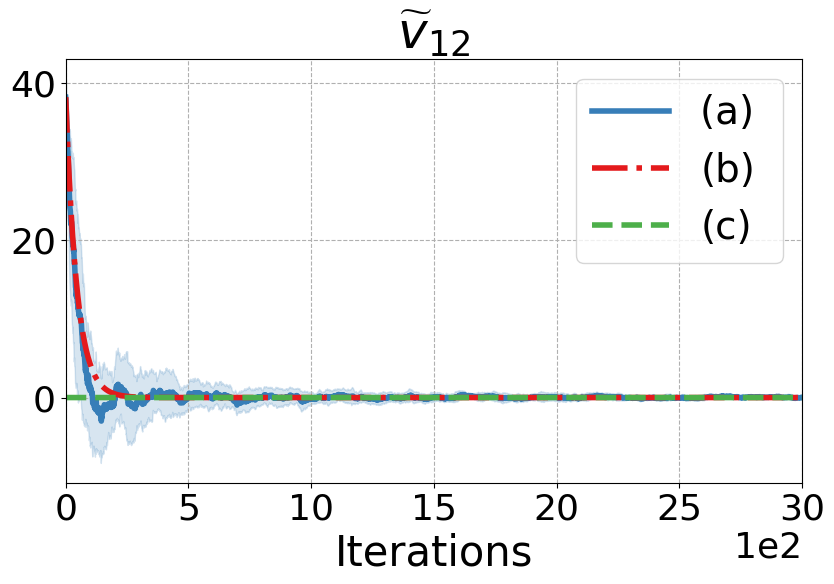}
    \includegraphics[width=0.32\textwidth]{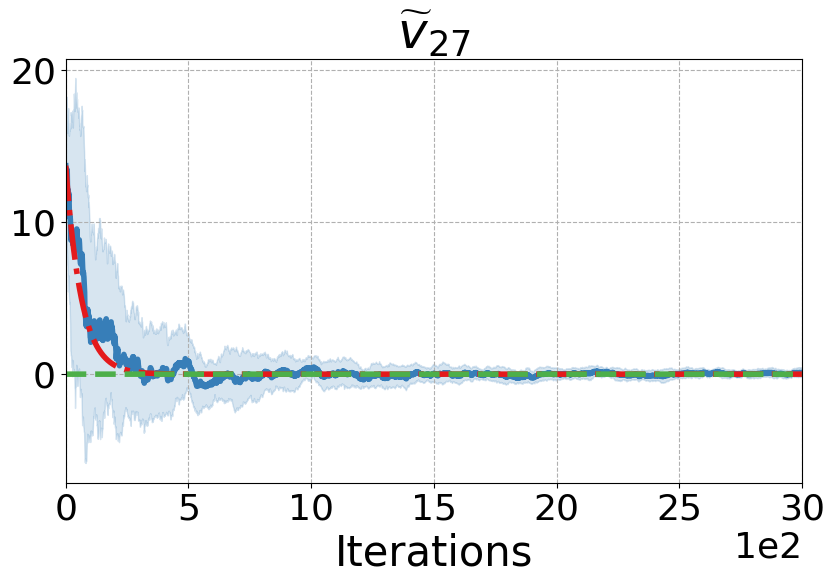}
    \includegraphics[width=0.32\textwidth]{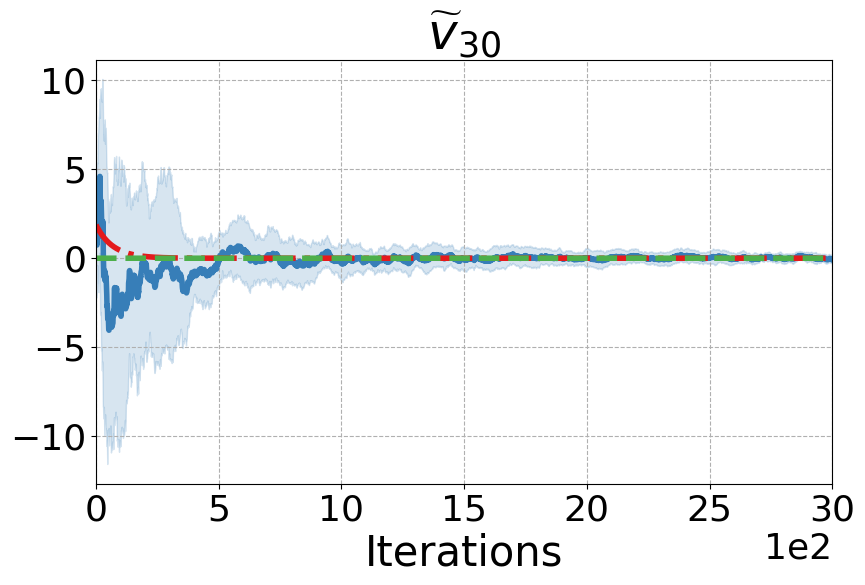}
    \includegraphics[width=0.32\textwidth]{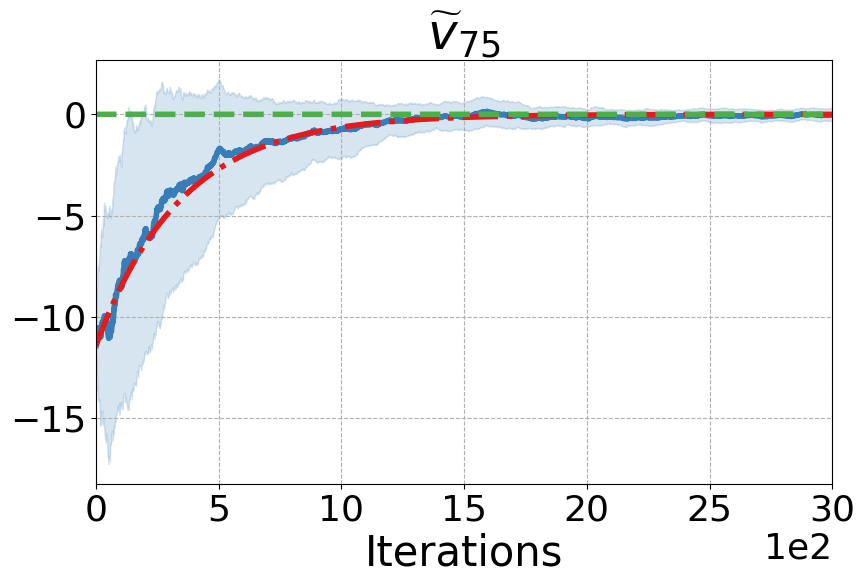}
    \includegraphics[width=0.32\textwidth]{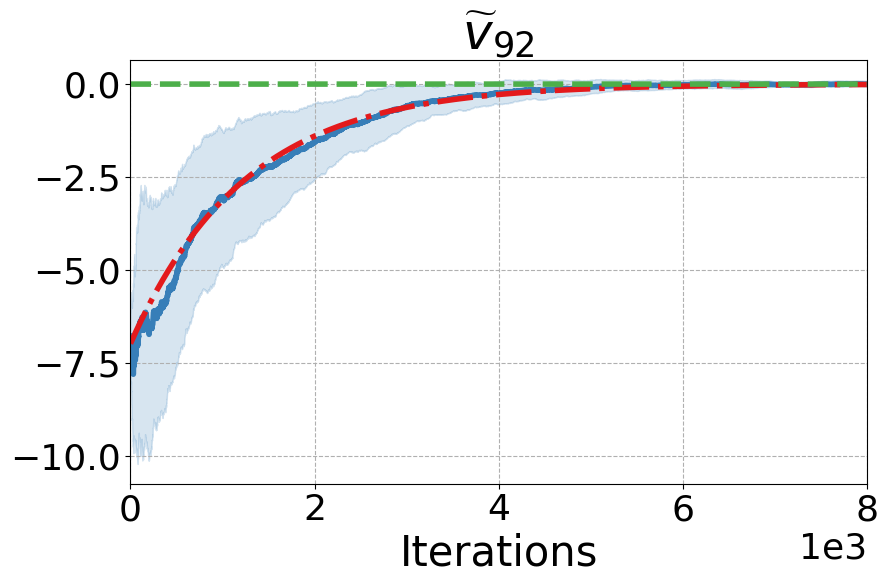}
    \includegraphics[width=0.32\textwidth]{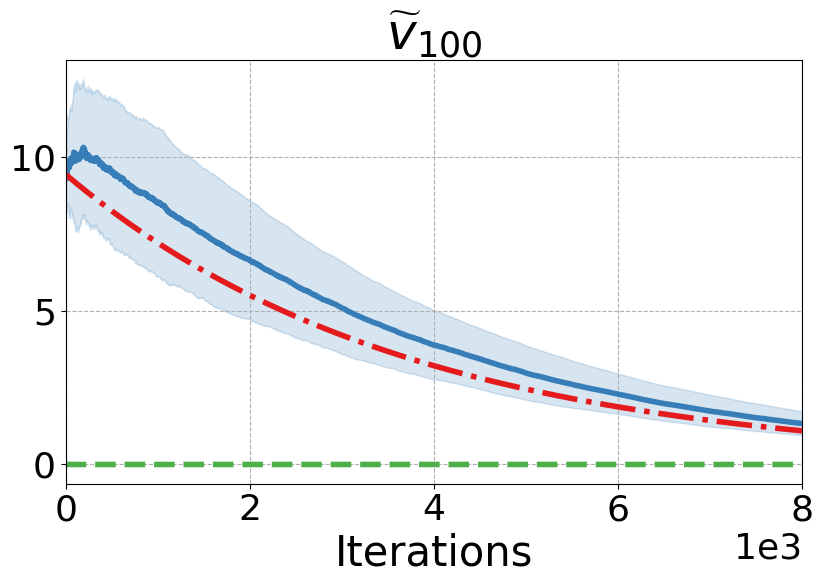}
    \caption{Quantities of the equality of Theorem \ref{thm:singular value additive noisy case} for multiple right singular vectors when $x_*=\widetilde x_{\rm LS}$. $(a)= \langle x_{k} - x_0^n-\widetilde x_{\rm LS},\widetilde v_j \rangle$ (averaged over $20$ runs), $(b)=\left(1-\mfrac{\widetilde{\sigma}_j^2}{\|\widetilde{A}\|_F^2}\right)^{k}  \langle  x_0^r-\widetilde x_{\rm LS},\widetilde v_j \rangle$, and $(c)=\left[1- \left(1-\mfrac{\widetilde{\sigma}_j^2}{\|\widetilde A\|_F^2}\right)^k\right] 
 \mfrac{\langle \widetilde{A} \widetilde x_{\rm LS} - \widetilde{b}, \widetilde u_j \rangle} 
{\widetilde{\sigma}_{j}}$. We have $m=1000$, $n=200$, ${\rm rank}(\widetilde A)=100$, and $x_0$ is arbitrary.}
  \label{fig:Thm3-2_xls}
\end{figure}

\bibliographystyle{plain}
\bibliography{references}

\end{document}